\documentclass[12pt,twoside,final]{article}

\usepackage{chngpage}
\usepackage[a4paper,left=22mm,right=22mm,top=30mm,bottom=30mm]{geometry}
\usepackage{imakeidx}
\usepackage{xr-hyper}
\usepackage{xr}
\usepackage[all]{xy}
\usepackage[center]{titlesec}
\titleformat*{\section}{\large\bfseries}
\usepackage{enumerate,cite}
\usepackage[latin1]{inputenc}
\usepackage{amssymb}
\usepackage{amsthm}
\usepackage{amsmath}
\usepackage{amsfonts}
\usepackage{mathrsfs}
\usepackage{hyperref}
\usepackage{color}
\usepackage{graphicx}
\usepackage{setspace}
\usepackage{bm}
\usepackage{enumitem}
\usepackage{booktabs}
\usepackage[capitalize]{cleveref}
\usepackage{tikz}
\usetikzlibrary{matrix,chains}
\usepackage[center]{titlesec}
\titleformat*{\section}{\normalsize\bfseries}
\usepackage{indentfirst}
\usepackage{amscd}

\hypersetup{
    linkbordercolor=0 1 0,
    citebordercolor=0 1 0}

\newcommand{\Aut}{\operatorname{Aut}\nolimits}
\newcommand{\Inn}{\operatorname{Inn}\nolimits}

\newcommand{\IBr}{\operatorname{IBr}\nolimits}

\newcommand{\Irr}{\operatorname{Irr}\nolimits}

\newcommand{\PSL}{\operatorname{PSL}}

\newcommand{\bl}{\operatorname{bl}}

\newcommand{\cE}{\mathcal{E}}

\newcommand{\bG}{\mathbf{G}}

\theoremstyle{remark}

\theoremstyle{definition}

\theoremstyle{plain}
\newtheorem{thm}{Theorem}[section]

\newtheorem{lem}[thm]{Lemma}

\newtheorem{prop}[thm]{Proposition}

\newtheorem{conj*}{Conjecture}
\numberwithin{equation}{thm}

\usepackage{lastpage}
\usepackage{fancyhdr}

\begin{document}

\begin{center}{\Large\bf Blocks with a single automorphism orbit of irreducible Brauer characters
}

\bigskip{Fuming Jiang

\scriptsize School of Mathematics and Statistics, Southwest University, Chongqing, 400715, China}

\medskip{Kun Zhang

\scriptsize Faculty of Mathematics and
Statistics, Hubei University, Wuhan 430062,
China.}

\medskip{Yuanyang Zhou

\scriptsize School of Mathematics and
Statistics, Central China Normal University,
Wuhan, 430079, P.R. China}
\end{center}

\bigskip{\noindent\small{\bf Abstract}
Fix a prime $p$.
Let $G\trianglelefteq\tilde G$ be finite groups such that
$|\tilde G:G|$ is a power of $p$, and let $\tilde b$ be a block of
$\tilde G$ with abelian defect groups
covering a block $b$ of $G$. We prove that $\tilde b$ is inertial whenever
$\IBr(b)$ is a single $\Aut(G)_b$-orbit. This verifies the abelian-defect case of the Kessar--Linckelmann conjecture. As a key ingredient, we further show that a block of a finite quasisimple group with a single automorphism orbit of irreducible Brauer characters has abelian defect groups. As applications, we prove
the blockwise Alperin weight conjecture for blocks with exactly one irreducible
Brauer character and establish Puig's nilpotency conjecture.
 }



\section  {Introduction}

Blocks with exactly one irreducible Brauer character form a natural
but still incompletely understood class in modular representation theory.
Their defining condition is character-theoretic, whereas the expected
conclusions concern the local and algebraic structure of the block.
Bridging the two aspects is closely related to several conjectures in block theory.

Kessar and Linckelmann conjectured that every block with exactly one irreducible Brauer character is Morita equivalent to a block of a $p$-solvable group \cite{MNS}.
Recall that a Morita equivalence between block algebras is \emph{basic} if it is induced by an endopermutation source bimodule \cite{P0}. Following \cite{P2}, a block is called \emph{inertial} if it is basically Morita equivalent to a block with a normal defect group. Inertiality provides precisely the
kind of description sought in the abelian-defect case of the
Kessar-Linckelmann conjecture.

Blocks with exactly one irreducible Brauer character are not well adapted to passage between
a group and its normal subgroups. This motivates a more general class of
blocks.
Let $p$ be a prime, and let $\mathcal O$ be a complete discrete valuation
ring with field of fractions of characteristic zero and algebraically closed
residue field of characteristic $p$.  We assume that the field of fractions
is sufficiently large for all groups under consideration.  For a block $b$
of a finite group $G$, write $\IBr(b)$ for the set of irreducible Brauer
characters belonging to $b$. Let $\Aut(G)_b$ denote the stabilizer of $b$ in $\Aut(G)$.  We
say that $\IBr(b)$ is a \emph{single $\Aut(G)_b$-orbit} if $\Aut(G)_b$ acts transitively on $\IBr(b)$.

The transitivity of the $\Aut(G)_b$-action on $\IBr(b)$ is weaker
than the condition $|\IBr(b)|=1$, but it still imposes strong
uniformity on the block $b$.  More importantly, it behaves well
under the block correspondences and normal-subgroup reductions that occur in
the proof.
Our main result is an extension theorem for blocks with abelian defect
groups.

\begin{thm}
\label{Main}
Let $\tilde{G}$ be a finite group containing $G$ as a normal subgroup of $p$-power index, let $\tilde{b}$ be a block of $\tilde{G}$ with abelian defect groups, and let $b$ be a block of $G$ covered by $\tilde{b}$. If $\IBr(b)$ is a single $\Aut(G)_b$-orbit, then $\tilde{b}$ is inertial.
\end{thm}

Taking $G=\tilde G$ shows that a block with abelian defect groups is
inertial whenever its irreducible Brauer characters form a single orbit
under its block stabilizer. In particular, \cref{Main} proves the
abelian-defect case of the Kessar--Linckelmann conjecture and Brou\'e's abelian defect group conjecture for such blocks
\cite{Br}.

We outline the proof of \cref{Main}, which proceeds in three stages. First, we reduce \cref{Main} to blocks of quasisimple groups with a single automorphism orbit of irreducible Brauer characters (see Proposition \ref{Reduction} below). Second,
for groups of Lie type in
 {non-defining} characteristic, we combine Jordan decomposition with
 {Bonnaf\'e--Rouquier} equivalences \cite{CE,KM2}, analyze corresponding blocks of Levi subgroups{,} and finally establish Proposition~\ref{All}. The remaining families are treated by their
known block structures. Finally, we combine the quasisimple analysis with the structural
results of \cite{P1} to show that the relevant covered blocks are
inertial, thereby completing the proof of Theorem~\ref{Main}.

The quasisimple analysis also yields a structural result that is useful
independently of \cref{Main}.

\begin{prop}\label{Quasi-one}
Let \(G\) be a finite quasisimple group and \(b\) a block of \(G\).
If \(\IBr(b)\) is a single \(\Aut(G)_{b}\)-orbit,
then defect groups of $b$ are abelian.
\end{prop}

This is a common generalization of the odd- and even-prime results on
nilpotent blocks of quasisimple groups in
\cite{AE, AE1}. Some cases of Proposition~\ref{Quasi-one} were already established
in \cite{MNS} (see the second paragraph of its proof); Proposition~\ref{All} completes the
remaining quasisimple cases.

Our next application concerns the blockwise Alperin weight conjecture. Recall that a \emph{$p$-weight} of a block $b$ of a finite group $G$ is a pair $(R,{\psi})$, where $R$ is a $p$-subgroup of $G$ such that $\mathbf{O}_p(N_G(R)) = R$ and where
$\psi$ is a {defect-zero}
irreducible character of $N_G(R)/R$ lying in a block $b_R$ of $N_G(R)$
inducing $b$.
The blockwise Alperin weight conjecture predicts that the number of irreducible Brauer characters in the block $b$
is equal to the number of orbits of the natural $G$-conjugation action on $p$-weights of $b$.

For blocks with exactly one irreducible Brauer character, the reduction
leads to blocks of quasisimple groups with a single automorphism orbit of irreducible Brauer characters.
Proposition~\ref{Quasi-one} shows that their defect groups are abelian,
and Theorem~\ref{Main} then implies that they are inertial.
Consequently, \cite{HZ} yields the following result.

\begin{thm}
\label{Main1}
The blockwise Alperin weight conjecture holds for blocks with exactly one irreducible Brauer character.
\end{thm}

Finally, we apply Theorem~\ref{Main1} to a nilpotency criterion
conjectured by Puig \cite{P-1}. He conjectured that a block
$b$ is nilpotent if for every $b$-Brauer pair
$(Q,e_Q)$, $e_Q$ has exactly one irreducible Brauer character. Watanabe related this criterion to the blockwise Alperin
weight conjecture \cite{Wa}. Combining {Watanabe's} argument with
Theorem~\ref{Main1} and induction on Brauer pairs proves Puig's
conjecture in full generality.

\begin{thm}
\label{Main2} Let $b$ be a block of a finite group $G$. Assume that
$e_Q$ has a unique irreducible Brauer character for every $b$-Brauer pair $(Q,e_Q)$.
Then $b$ is nilpotent.
\end{thm}

The paper is organized as follows.  Section~2 develops the normal-subgroup
and extension machinery and reduces \cref{Main} to quasisimple groups.
Section~3 establishes the required classification result for quasisimple
groups, with the finite reductive groups treated through Jordan
correspondence and Levi subgroups.  Section~4 completes the proof of
\cref{Main} and derives Proposition~\ref{Quasi-one}.  In Section~5 we prove \cref{Main1} and then
establish \cref{Main2}.

\section{Reduction of Theorem \ref{Main}}

In this section, we reduce the proof of Theorem~\ref{Main} to quasisimple groups (see Proposition~\ref{Reduction} below).

\begin{lem}\label{Fong}
Let $b$ be a block of a finite group $G$, let $N \trianglelefteq G$, and let $c$ be a block of $N$ covered by $b$.
Let $G_c$ denote the stabilizer of $c$ in $G$, and let $f$ be the Fong--Reynolds correspondent of $b$ in $G_c$.
Assume that $N$ is $\Aut(G)_b$-invariant and that $\IBr(b)$ is a single $\Aut(G)_b$-orbit.
Then $\IBr(c)$ is a single $\Aut(N)_c$-orbit and $\IBr(f)$ is a single $\Aut(G_c)_f$-orbit.
\end{lem}

\begin{proof}
Let $\varphi_1, \varphi_2 \in \IBr(c)$ and choose $\chi_i \in \IBr(b)$ lying over $\varphi_i$.
Since $\IBr(b)$ is a single $\Aut(G)_b$-orbit, there exists $\alpha \in \Aut(G)_b$ such that $\chi_1^\alpha = \chi_2$.
As $N$ is $\Aut(G)_b$-invariant, $\alpha(N) = N$ and hence $\varphi_1^\alpha$ and $\varphi_2$ both lie below $\chi_2$.
By \cite[Corollary 8.7]{N1}, there exists an inner automorphism $\tau$ of $G$ such that $\varphi_1^{\alpha\circ\tau} = \varphi_2$.
Hence the restriction of $\alpha\circ\tau$ to $N$ lies in $\Aut(N)_c$ and $\IBr(c)$ is a single $\Aut(N)_c$-orbit.

Now let $\varphi_1, \varphi_2 \in \IBr(f)$. By \cite[Theorem 9.14]{N1}, the induced characters $\chi_i = \varphi_i^G$ belong to $\IBr(b)$.
Since $\IBr(b)$ is a single $\Aut(G)_b$-orbit, there exists $\alpha \in \Aut(G)_b$ with $\chi_1^\alpha = \chi_2$.
By \cite[Corollary 9.3]{N1}, there exists an inner automorphism $\tau$ of $G$ such that $c^{\alpha\circ\tau} = c$, and therefore $G_c^{\alpha\circ\tau} = G_c$.
Applying \cite[Theorem 9.14]{N1} again yields $\varphi_1^{\alpha\circ\tau} = \varphi_2$, so the restriction of $\alpha\circ\tau$ to $G_c$ belongs to $\Aut(G_c)_f$ and $\IBr(f)$ is a single $\Aut(G_c)_f$-orbit.
\end{proof}

\begin{lem}\label{Normal-P-cap-N}
Let $b$ be a block of a finite group $G$ with defect group $P$ and let $N \trianglelefteq G$ be $\Aut(G)_b$-invariant.
Let $c$ be a $G$-invariant block of $N$ covered by $b$ and set $D = P \cap N$.
Then the subgroup $C_G(D)N$ is $\Aut(G)_b$-invariant.
\end{lem}

\begin{proof}
The block $c$ is the unique block of \(N\) covered by \(b\) since it is \(G\)-invariant.
By \cite[Theorem 9.26]{N1}, \(D\) is a defect group of \(c\).
The uniqueness implies that \(c\) is \(\Aut(G)_b\)-invariant, so
\(D^\alpha\) is a defect group of \(c\) for any \(\alpha \in \Aut(G)_b\).
Since \(N\) acts transitively on the set of defect groups of \(c\),
there exists \(n \in N\) such that \( D^\alpha = D^n\). Consequently we have
$C_G(D)^\alpha= C_G(D^\alpha) = C_G(D^n)=C_G(D)^n$ and therefore
$
(C_G(D)N)^\alpha = C_G(D)^\alpha N = C_G(D)^nN = C_G(D)N
$.
\end{proof}

\begin{lem}\label{Abel}
Let \(b\) be a block of a finite group \(G\) with a defect group \(P\) and let \(N\) be a characteristic subgroup of \(G\).
Assume that all blocks of all characteristic subgroups of \(G\) containing \(N\) covered by \(b\) are \(G\)-invariant and that the quotient group \(G/N\) is \(p\)-solvable.
Then {the subgroup} \(PN/N\) is a Sylow \(p\)-subgroup of \(G/N\).
\end{lem}

\begin{proof}
We proceed by an argument analogous to that in the proof of \cite[Lemma 2.4]{Ar}.
Take the full inverse images in $G$ of the terms of the upper $p$-series of
$G/N$
\[
  N=N_0\leq N_1\leq\cdots\leq N_t=G.
\]
The subgroups $N_i$ are characteristic in $G$, and each factor
$N_{i+1}/N_i$ is either a $p$-group or a $p'$-group.
For each \(i\), let \(b_i\) be a block of \(N_i\) covered by \(b\) and set $P_i = P \cap N_i$. By \cite[Theorem 9.26]{N1},
\(P_i = P \cap N_i\) is a defect group of \(b_i\).
If \(|N_{i+1}:N_i|\) is coprime to \(p\), then \(P_i=P_{i+1}\).
If \(|N_{i+1}:N_i|\) is a power of \(p\), then the invariance hypothesis for $b_i$ and
\cite[Theorem 9.17]{N1} yield \(|P_{i+1} : P_i| = |N_{i+1} : N_i|\).
Consequently, multiplication along the series shows
$
|PN/N| = |G/N|_p.
$
\end{proof}

We denote by $\mathbf{Z}(G)$ the center of $G$, by $\mathbf{F}(G)$ the Fitting subgroup of \(G\), by \(\mathbf{E}(G)\)
the layer of \(G\) and by \(\mathbf{F}^{*}(G)\) the generalized Fitting subgroup of $G$.

\begin{lem}\label{Solv}
Assume that \(\mathbf{F}^*(G) = \mathbf{Z}(G)\mathbf{E}(G)\) and that all components of \(G\) are normal. Then \(G/\mathbf{F}^*(G)\) is solvable.
\end{lem}

\begin{proof} Let $X_1,\dots,X_r$ be the components of $G$. The $G$-conjugation induces a group homomorphism
$\pi: G\rightarrow \prod_{i=1}^r \Aut(X_i)$. Since \(\mathbf{F}^*(G) = \mathbf{Z}(G)\mathbf{E}(G)\),
the kernel of $\pi$ is $C_G(\mathbf{F}^*(G))$. Moreover, $\pi$ maps $\mathbf{F}^*(G)$ onto
$\prod_{i=1}^r \mathrm{Inn}(X_i)$ and therefore $\pi$ induces an injective homomorphism
$G/\mathbf{F}^*(G) \rightarrow \prod_{i=1}^r \mathrm{Out}(X_i)$.
By the classification of finite simple groups, each $\mathrm{Out}(X_i)$ is solvable and hence $G/\mathbf{F}^*(G)$ is solvable.
\end{proof}

\begin{lem}\label{Normal-one-orbit2}
Let \(G\) be a finite group and \(b\) a block of \(G\).
Assume that \(\IBr(b)\) is a single \(\Aut(G)_b\)-orbit.
Then for every component \(X\) of \(G\) and every block \(f\) of \(X\) covered by a block of \(\mathbf{E}(G)\) covered by \(b\), \(\IBr(f)\) is a single \(\Aut(X)_f\)-orbit.
\end{lem}

\begin{proof}
Let \(d\) be a block of \(\mathbf{E}(G)\) covered by \(b\) and covering \(f\).  By Lemma \ref{Fong}, \(\IBr(d)\) is a single \(\Aut(\mathbf{E}(G))_d\)-orbit. Let \(\{X_1, \cdots, X_r\}\) be the orbit of the natural action of $\Aut(\mathbf{E}(G))_d$ on the components of $G$ such that $X_1=X$.
Set \(L=X_1X_2 \cdots X_r\) and
let \(d_L\) be a block of \(L\) covered by \(d\).
Lemma \ref{Fong} shows that \(\IBr(d_L)\) is a single \(\Aut(L)_{d_L}\)-orbit.

For each $i$, choose
$\alpha_i\in\Aut(\mathbf E(G))_d$ such that
$\alpha_1=1$ and $X_i=\alpha_i(X)$.
Since $X_1, \ldots, X_r$ {centralize one another}, the block $d_L$ of $L$ covered by $d$
is unique. Hence $d_L$ is stabilized by every $\alpha_i$ and
$f_i=f^{\alpha_i}$ is the block of $X_i$ covered by $d_L$. Denote by $X^r$ the direct product of $r$ copies of $X$ and
define $\pi:X^r\rightarrow L, (x_1,\ldots,x_r)\mapsto \alpha_1(x_1)\cdots\alpha_r(x_r)$.
Since distinct components commute and intersect centrally,
$\pi$ is surjective and $\ker(\pi)\leq Z(X^r)$. Let $e$ be the unique block of $X^r$ dominating $d_L$.

There is an $\mathcal O$-algebra isomorphism
\(
 \mathcal OX^r\cong\bigotimes_{i=1}^r\mathcal OX
\)
mapping $(x_1,\ldots,x_r)$ to $x_1\otimes\cdots\otimes x_r$. Identifying $\mathcal OX^r$ and $\bigotimes_{i=1}^r\mathcal OX$, we can write $e=e_1\otimes\cdots\otimes e_r$, where each $e_i$ is a block of $X$.
We denote by $\pi_i$ the homomorphism $X\rightarrow X^r$ mapping $x$ to the element of $X^r$ whose $i^{\rm th}$-component is $x$ and whose other components are $1$. The block $e_i$ is the block of the $i$-th direct factor corresponding to $e$.

We claim that each $e_i$ is equal to $f$. The composition $\pi\pi_i$ is the injective homomorphism $X\rightarrow L$, $x\mapsto\alpha_i(x)$. The block $e_i$ corresponds to the block $e_i^{\alpha_i}$ of $X_i=\alpha_i(X)$, and $d_L$ covers $e_i^{\alpha_i}$. Hence $d$ covers $e_i^{\alpha_i}$. Since $d$ also covers $f_i$ and the distinct components commute, we have $e_i^{\alpha_i}=f_i$.
Since $f^{\alpha_i}=f_i$, we have $e_i=f$. This proves the claim.

Let $\nu: \tilde X\to X$ be the universal central extension and
let $\tilde f$ be the unique block of $\tilde X$
dominating $f$. Then $\tilde e=\tilde f\otimes \tilde  f\otimes \cdots\otimes \tilde f$
is the unique block of $\tilde X^r$ dominating $d_L$.
Let $\rho$ be the homomorphism
$\tilde X^r\rightarrow L,
 (\tilde x_1,\ldots,\tilde x_r)\mapsto
 \alpha_1(\nu(\tilde x_1))\cdots
 \alpha_r(\nu(\tilde x_r))
$ and set $M=\ker(\rho)$.  By \cite[Lemma B.8]{N2},
$\Aut(L)_{d_L}$ lifts to the subgroup
\[
 A=\Aut(\tilde X^r)_{\tilde e,M}
   =\{\gamma\in\Aut(\tilde X^r)\mid
       \tilde e^\gamma=\tilde e,\ \gamma(M)=M\}.
\]
Since \(\IBr(d_L)\) is a single \(\Aut(L)_{d_L}\)-orbit,  \cite[Corollary B.8]{N2} yields that \(\IBr(\tilde e)\) is a single \(A\)-orbit.

Let $\varphi,\psi\in\IBr(f)$ and denote their inflations to
$\tilde X$ by $\tilde\varphi$ and
$\tilde\psi$ respectively. Then $\tilde\varphi$ and
$\tilde\psi$ {lie} in the block $\tilde f$. Denote by $\tilde\varphi^{\otimes r}$ the tensor product of $r$ copies of $\tilde\varphi$, which is an irreducible Brauer character of $\tilde X^r$ lying in the block {$\tilde e$}.
There exists $\gamma\in A$
such that $(\tilde\varphi^{\otimes r})^\gamma=\tilde\psi^{\otimes r}$.
By \cite[Lemma 10.24]{N2}, write
$\gamma=(\gamma_1,\ldots,\gamma_r)\sigma$ with
$\gamma_i\in\Aut(\tilde X)$ and
$\sigma\in\mathfrak S_r$, where \(\mathfrak{S}_r\) is the symmetric group of degree \(r\).

If $U_i$ denotes the $i$-th
direct factor of $\tilde X^r$, then the $i$-th projection $\tilde X^r\rightarrow X$ induces an isomorphism $
 M\cap U_i\cong \ker(\nu)$.
Since $\gamma(M)=M$ and $\gamma(U_i)=U_{\sigma(i)}$, each
$\gamma_i$ stabilizes $\ker(\nu)$ and hence by
\cite[Lemma B.8]{N2}, induces an automorphism
$\overline{\gamma}_i\in\Aut(X)$.  Moreover,
$\gamma\in\Aut(\tilde X^r)_{\tilde e}$ implies
$\gamma_i\in\Aut(\tilde X)_{\tilde f}$, and therefore
$\overline{\gamma}_i\in\Aut(X)_f$.  Comparing tensor factors
gives
$\tilde\varphi^{\gamma_i}=\tilde\psi$ for some $i$ and hence
$
 \varphi^{\overline{\gamma}_i}=\psi$.
Thus $\IBr(f)$ is a single $\Aut(X)_f$-orbit.
\end{proof}

Let
$G\trianglelefteq\tilde G$ with $\tilde G/G$ a $p$-group and let $b$ be a $\tilde G$-invariant block of $G$. Note that the induced block $b^{\tilde G}$ of $b$ in $\tilde G$ exists and is equal to the block $b$.
Choose defect groups $P$ of $b$ and $\tilde P$ of $b^{\tilde G}$ so that
$\tilde P\cap G=P$. We have
\(\tilde{G}=\tilde{P}G\). The $\tilde G$-conjugation induces a homomorphism
\(\tau \colon \tilde{G} \to \Aut(G)\).
We define the {semidirect} products
$G^\tau = G \rtimes \tau(\tilde{P})$ and $P^\tau= P \rtimes \tau(\tilde{P})$.

\begin{lem}\label{Defect}
With the above notation, \(P^\tau\) is a defect group of the induced block $b^{G^\tau}$. Moreover, if \(\tilde {P}\) is abelian, then the block \(b^{\tilde G}\) is inertial if and only if the block \(b^{G^\tau}\) is inertial.
\end{lem}

\begin{proof}
Regard $\mathcal{O}G$ as a $\tilde G$-algebra via the conjugation. Given a $p$-subgroup $Q$ of $\tilde G$, denote by ${\rm Br}_{\mathcal{O} G}^Q$ the associated Brauer homomorphism. Since the index of $G$ in $\tilde G$ is a $p$-power, $Q$ is maximal such that ${\rm Br}_{\mathcal{O} \tilde G}^{Q}(b^{\tilde G})\neq 0$ if and only if $Q$ is maximal such that ${\rm Br}_{\mathcal{O} G}^{Q}(b)\neq 0$. On the other hand, it follows from the transitivity of Brauer homomorphisms that we have $${\rm Br}_{\mathcal{O} G}^{P^\tau}(b)={\rm Br}_{k C_G(P)}^{P^\tau}({\rm Br}_{\mathcal{O} G}^P(b))={\rm Br}_{k C_G(P)}^{\tilde P}({\rm Br}_{\mathcal{O} G}^P(b))={\rm Br}_{\mathcal{O} G}^{\tilde P}(b).$$ Consequently, the maximality of $\tilde P$ such that ${\rm Br}_{\mathcal{O} G}^{\tilde P}(b)\neq 0$ forces the maximality of $P^\tau$ such that ${\rm Br}_{\mathcal{O} G}^{P^\tau}(b)\neq 0$.
Therefore $P^\tau$ is maximal such that ${\rm Br}_{\mathcal{O} G^\tau}^{P^\tau}(b^{G^\tau})\neq 0$ and $P^\tau$ is a defect group of $b^{G^\tau}$.

Suppose that \(\tilde {P}\) is abelian. Then $P^\tau=P\times\tau(\tilde P)$ is abelian.
Consider the semidirect product $\hat G=G\rtimes \tilde P$ determined by the $\tilde P$-conjugation and set $Q=\{(x, x^{-1})|x\in P\}$. The map $\hat G\rightarrow \tilde G, (x, \tilde x)\mapsto x\tilde x$ is a surjective group homomorphism whose kernel $Q$ is central in $\hat G$.
As in the previous paragraph, we prove that $b^{\hat G}$ has defect group $P\times \tilde P$.
Obviously $b^{\hat G}$ dominates $b^{\tilde G}$. By \cite[Corollary 1.14]{P1}, $b^{\hat G}$ is inertial if and only if so is $b^{\tilde G}$. Similarly,
the map $\hat G\rightarrow G^\tau, (x, \tilde x)\mapsto (x, \tau(\tilde x))$ is a surjective homomorphism with kernel the central $p$-subgroup $\{(1, \tilde x)|\tilde x\in \tilde P\, \mbox{\rm and}\, \tau(\tilde x)=1\}$,
$b^{\hat G}$ dominates $b^{G^\tau}$, and $b^{\hat G}$ is inertial if and only if so is $b^{G^\tau}$.
Combining these equivalences, we conclude that $b^{\tilde G}$ is inertial if and only if so is $b^{G^\tau}$.
\end{proof}

\begin{lem}\label{Normal-one-orbit}
Keep the notation of Lemma~\ref{Defect}. Let $N$ be an
\(\Aut(G)_b\)-invariant subgroup of $G$, and let $c$ be a
$G$-invariant nilpotent block of $N$ covered by $b$. Set
\[
 D=P\cap N,\qquad E=G/N,\qquad
 E^\tau=E\rtimes\tau(\tilde P).
\]
Then there exist a finite group $L^\tau$, a normal subgroup
$L\trianglelefteq L^\tau$, and a surjective homomorphism
$\rho:L^\tau\to E^\tau$ with kernel $D$, satisfying the
following properties.

\smallskip\noindent{\bf 2.7.1.}
There is an injective homomorphism $P^\tau\hookrightarrow L^\tau$, which we use
to identify $P^\tau$ with its image, such that the composite $
 P^\tau\rightarrow L^\tau\stackrel{\rho}{\rightarrow}E^\tau
$
is the natural map with kernel $D$. In particular, the subgroup \(\tau(\tilde P)\) of $P^\tau$ is mapped identically onto
the corresponding subgroup of $E^\tau$.

\smallskip\noindent{\bf 2.7.2.}
We have $L=\rho^{-1}(E)$, $P^\tau\cap L=P$, and $\rho$
induces an isomorphism $L/D\cong E$. Moreover, we have
$
 L^\tau=L\rtimes\tau(\tilde P).
$

\smallskip\noindent{\bf 2.7.3.}
There is a central extension
$
 1\rightarrow Z\rightarrow\tilde L^\tau
   \stackrel{\eta}{\rightarrow}L^\tau\rightarrow1,
$
where $Z$ is a finite subgroup of $k^*$. Set
\(\tilde L=\eta^{-1}(L)\). There is a
\(\tilde L^\tau\)-invariant block $d$ of \(\tilde L\)
which is also a block of \(\tilde L^\tau\), such that
$$
 \mathcal O G b\simeq_{\mathrm{bM}}\mathcal O\tilde Ld
\quad {\rm and} \quad
 \mathcal O G^\tau b^{G^\tau}\simeq_{\mathrm{bM}}
 \mathcal O\tilde L^\tau d^{\tilde L^\tau},$$
where \(\simeq_{\mathrm{bM}}\) denotes a basic Morita equivalence.
Furthermore, $Z$ has $p'$-order and a complement to $Z$ in
\(\eta^{-1}(P^\tau)\) can be chosen to be a defect group of $d^{\tilde L^\tau}$.
We identify this complement with $P^\tau$.

\smallskip\noindent{\bf 2.7.4.}
If \(\IBr(b)\) is a single \(\Aut(G)_b\)-orbit, then
\(\IBr(d)\) is a single \(\Aut(\tilde L)_d\)-orbit.
\end{lem}

\begin{proof}
Set \(\mathcal A=\Aut(G)_b\), \(\mathcal G=G\rtimes\mathcal A\), and \(\mathcal E=E\rtimes\mathcal A\).
Since $c$ is $G$-invariant, it is the unique block of $N$
covered by $b$. The subgroup $N$ and the block $b$ are
\(\mathcal A\)-invariant, and hence so is $c$. Moreover,
$D=P\cap N$ is a defect group of $c$.
The structure theorem for extensions of nilpotent blocks
\cite[Theorem~3.5 and Corollary~3.14]{PZ} yields an extension
\[
 1\rightarrow D\rightarrow\mathcal L
   \rightarrow\mathcal E\rightarrow1
\]
such that $\mathcal L$ contains $P^\tau$ and the quotient map $\mathcal L
   \rightarrow\mathcal E$ preserves \(\tau(\tilde P)\) elementwise; moreover, it also yields a class
\(\alpha\in H^2(\mathcal L,k^*)\) and, in the terminology of
\cite[Section~9]{CMT}, an \(\mathcal E\)-graded basic Morita
equivalence
\[
 \mathcal O\mathcal Gc
   \simeq_{\mathrm{bM}}
 \mathcal O_{\alpha}\hat{\mathcal L}.
\leqno{2.7.5}
\]
By \cite[Proposition~1.2.18]{L}, the $k^*$-group
\(\hat{\mathcal L}\) contains a finite subgroup
\(\tilde{\mathcal L}\) such that
\(\hat{\mathcal L}=k^*\tilde{\mathcal L}\). If
\(Z=k^*\cap\tilde{\mathcal L}\), then, for a suitable primitive
idempotent $e\in\mathcal OZ$, $\mathcal O_{\alpha}\hat{\mathcal L}$ is
isomorphic to
\(\mathcal O\tilde{\mathcal L}e\) as an \(\mathcal E\)-graded algebra. In particular, lifting $D$ to a subgroup of $\tilde{\mathcal L} $ and identifying $D$ with the subgroup, we have
\[
 \tilde{\mathcal L}/DZ\cong\mathcal E.
\leqno{2.7.6}
\]

Let $L$ and $L^\tau$ be the inverse images of $E$ and
$E^\tau$, respectively, in \(\mathcal L\), and let
$\tilde L$ and \(\tilde L^\tau\) be their full inverse
images in \(\tilde{\mathcal L}\). The construction of
\(\mathcal L\), together with the injective homomorphism $P^\tau\rightarrow \tilde{\mathcal L}$,
gives Statements~1 and~2. Restricting 2.7.5 to the homogeneous
components indexed by $E$ and $E^\tau$ gives, respectively,
graded basic Morita equivalences
\[
 \mathcal OGc\simeq_{\mathrm{bM}}\mathcal O\tilde Le
 \quad\text{and}\quad
 \mathcal OG^\tau c\simeq_{\mathrm{bM}}
       \mathcal O\tilde L^\tau e.
\leqno 2.7.7
\]
Let $d$ be the block of \(\tilde L\) corresponding to $b$
under the first equivalence in 2.7.7. Since $b$ is
\(G^\tau\)-invariant and $L^\tau/L$ is a $p$-group, $d$ is
\(\tilde L^\tau\)-invariant and is also a block of
\(\tilde L^\tau\). The second equivalence in 2.7.7 restricts to
a basic Morita equivalence
\[
 \mathcal OG^\tau b^{G^\tau}
   \simeq_{\mathrm{bM}}
 \mathcal O\tilde L^\tau d.
\leqno 2.7.8
\]
By Lemma~\ref{Defect}, $P^\tau$ is a defect group of
\(b^{G^\tau}\).
Since $Z\leq k^*$ is finite, $Z$ has $p'$-order and hence
\(\eta^{-1}(P^\tau)\) has a complement to $Z$. According to \cite[Corollary~3.14]{PZ} and the choice of Morita equivalence in 2.7.5,
this complement may be chosen as a defect group of $d$. This proves
Statement~3.

It remains to prove Statement~4. Assume that \(\IBr(b)\) is a single
\(\mathcal A\)-orbit. Choose \(\varphi\in\IBr(b)\) and an
irreducible Brauer character \(\chi\in\IBr(\mathcal G)\) lying over
\(\varphi\). Clifford theory and the transitivity of \(\mathcal A\)
give
\[
 \chi_G=m\sum_{\theta\in\IBr(b)}\theta
\leqno 2.7.9
\]
for some positive integer $m$. Let \(\psi\) be the irreducible
Brauer character of \(\tilde{\mathcal L}\) corresponding to
\(\chi\) under 2.7.5. By \cite[Theorem~5.1.18(a)]{Marcus}, the equality in 2.7.9 corresponds to the equality
\[
 \psi_{\tilde L}
   =m\sum_{\eta\in\IBr(d)}\eta.
\]
By Clifford theory, the irreducible constituents of
\(\psi_{\tilde L}\) form a single
\(\tilde{\mathcal L}\)-orbit. Thus conjugation by
\(\tilde{\mathcal L}\) is transitive on \(\IBr(d)\); it also
stabilizes $d$ and induces automorphisms of \(\tilde L\).
Consequently, \(\IBr(d)\) is a single
\(\Aut(\tilde L)_d\)-orbit.
\end{proof}

\begin{prop}\label{Reduction}
Assume that Theorem~\ref{Main} holds whenever $G$ is quasisimple. Then Theorem~\ref{Main} holds for arbitrary finite group $G$.
\end{prop}

\begin{proof}
We continue to employ the notation and {assumptions} in Theorem~\ref{Main}.
We aim to show that \(\tilde{b}\) is inertial. Denote by $\tilde G_b$ the stabilizer of $b$ in $\tilde G$.
Since the index $|\tilde G: G|$ is a $p$-power, we have $b^{\tilde G}=\tilde b$ and $b^{\tilde G_b}=b$.
Clearly $b^{\tilde G}$ is the Fong-Reynolds correspondent of $b^{\tilde G_b}$ in $\tilde G$ and they have common defect groups and are basically Morita equivalent. Hence if $b^{\tilde G_b}$ is inertial, $\tilde b$ is inertial since basic Morita equivalence preserves the inertiality of blocks. So we may
assume that the block $b$ is $\tilde G$-invariant, so that $\tilde b=b^{\tilde G_b}$.
Choose a defect group $\tilde P$ of the block $\tilde b$ and set $P=\tilde P\cap G$. Since the index of $G$ in $\tilde{G}$ is a power of $p$, it follows that $\tilde G=\tilde P G$ and $P$ is a defect group of the block $b$.

Suppose that Theorem~\ref{Main} does not hold in general, and choose a
counterexample $(\tilde G, G, b)$ for which
$|G:\mathbf{O}_{p'}(\mathbf{Z}(G))|$
is the smallest. Let $H$ be an $\Aut(G)_b$-invariant subgroup of
$G$, and let $c$ be a block of $H$ covered by $b$. Since
$\mathrm{Inn}(G)\leq\Aut(G)_b$, the subgroup $H$ is normal in
$G$. Moreover, conjugation by $\tilde G$ induces automorphisms of $G$
stabilizing $b$, hence $H$ is normal in $\tilde G$. Write $G_c$
and $\tilde G_c$ for the stabilizers of $c$ in $G$ and $\tilde G$,
respectively. Let $e$ and $\tilde e$ be the Fong--Reynolds
correspondents of $b$ in $G_c$ and $\tilde G_c$, respectively. We have
$G_c=G\cap\tilde G_c\trianglelefteq\tilde G_c$, and
$\tilde G_c/G_c$ is a $p$-group. Since both $\tilde e$ and $e$ cover $c$, $\tilde e$ covers $e$.
Moreover, $\tilde e$ has abelian defect groups because the
Fong-Reynolds correspondence preserves defect groups. By
Lemma~\ref{Fong}, $\IBr(e)$ is a single
$\Aut(G_c)_e$-orbit.

We claim that $G_c=G$. Indeed,
$\mathbf{O}_{p'}(\mathbf{Z}(G))\leq
\mathbf{O}_{p'}(\mathbf{Z}(G_c))$. Thus, if
$G_c<G$, then we have
\[
 |G_c:\mathbf{O}_{p'}(\mathbf{Z}(G_c))|
 \leq |G_c:\mathbf{O}_{p'}(\mathbf{Z}(G))|
 < |G:\mathbf{O}_{p'}(\mathbf{Z}(G))|.
\]
The {minimality} of the counterexample would then imply that
$\tilde e$ is inertial. Since $\mathcal{O}\tilde G_c\tilde e$ and $\mathcal{O}\tilde G\tilde b$
are basically Morita equivalent, $\mathcal{O}\tilde G\tilde b$ is inertial, a
contradiction. Hence $G_c=G$, so $c$ is $G$-invariant and unique.
It follows that $c$ is $\Aut(G)_b$-invariant and also
$\tilde G$-invariant. We denote it by $b_H$. Finally,
Lemma~\ref{Fong} shows that $\IBr(b_H)$ is a single
$\Aut(H)_{b_H}$-orbit, and
\cite[Theorem~9.26]{N1} shows that $P\cap H$ is a defect group of
$b_H$.

\medskip\noindent
\textbf{Step 1.} Assume that \(b_H\) is nilpotent. Then \(H \leq \mathbf{Z}(G)\mathbf{O}_p(G)\). In particular, \(\mathbf{O}_{p'}(G) \leq \mathbf{Z}(G)\).

\smallskip
Set \(N=H\mathbf{Z}(G)\). Then \(N\) is
\(\Aut(G)_b\)-invariant. By the preceding reduction, there is
a unique block \(b_N\) of \(N\) covered by \(b\), and
\(D=P\cap N\) is a defect group of \(b_N\). By transitivity of
block covering, \(b_N\) covers \(b_H\) and a block of
\(\mathbf{Z}(G)\). Since $b_H$ is nilpotent, it follows from \cite[Lemma~2.7(3)]{Ar} that \(b_N\) is nilpotent.

Apply Lemma~\ref{Normal-one-orbit} to \(N\) and \(c=b_N\) and {use}
the notation of that lemma. Thus \(\tilde L\trianglelefteq\tilde
L^\tau\), the quotient \(\tilde L^\tau/\tilde L\) is a \(p\)-group,
and \(d\) is a block of both \(\tilde L\) and \(\tilde L^\tau\).
We first claim that \((\tilde L^\tau,\tilde L,d)\) is a
counterexample to Theorem~\ref{Main}. It follows from Lemma~\ref{Normal-one-orbit}(4) that
\(\IBr(d)\) is a single \(\Aut(\tilde L)_d\)-orbit.
Lemma~\ref{Normal-one-orbit}(3) shows that \(P^\tau\) is a defect group of \(d\), which is abelian since $\tilde P$ is abelian. Since $\tilde b$ is not inertial, by Lemma~\ref{Defect}, the block $b^{G^\tau}$ is not inertial.
Lemma~\ref{Normal-one-orbit}(3) gives a basic Morita equivalence
between $b^{G^\tau}$ and $d^{\tilde L^\tau}$. Since inertiality is
preserved under basic Morita equivalence, $d^{\tilde L^\tau}$ is not
inertial. Hence $(\tilde L^\tau,\tilde L,d)$ is a counterexample to
Theorem~\ref{Main}. {This proves the claim.}

Set \(Z_0=\mathbf{O}_{p'}(\mathbf{Z}(G))\) and let \(Z\) be the
central \(p'\)-subgroup occurring in
Lemma~\ref{Normal-one-orbit}(3). We may identify \(D\) with its lift
to \(\tilde L\). Since
$Z\leq \mathbf{O}_{p'}(\mathbf{Z}(\tilde L))$ and $\tilde L/DZ\cong G/N$,
we obtain
\[
 |\tilde L:\mathbf{O}_{p'}(\mathbf{Z}(\tilde L))|
 \leq |\tilde L:Z|\\
 =|G:N|\,|D|\\
 \leq |G:Z_0|.
\]
For the last inequality, note that \(DZ_0\leq N\) and
\(D\cap Z_0=1\), so that \(|D|\,|Z_0|\leq |N|\). Since
\((\tilde L^\tau,\tilde L,d)\) is a counterexample, the minimal choice
of \((\tilde G,G,b)\) gives the reverse inequality
\[
 |G:Z_0|
 \leq |\tilde L:\mathbf{O}_{p'}(\mathbf{Z}(\tilde L))|.
\]
Consequently, equality holds throughout the preceding chain. In
particular, \(|N|=|D|\,|Z_0|\), and hence
\(N=DZ_0=D\times Z_0\). Thus \(D\) is the unique Sylow
\(p\)-subgroup of \(N\). Since \(N\trianglelefteq G\), it follows
that \(D\trianglelefteq G\), and therefore \(D\leq\mathbf{O}_p(G)\).
{Therefore},
\[
 H\leq N=DZ_0\leq \mathbf{O}_p(G)\mathbf{Z}(G).
\]

Finally, take \(H=\mathbf{O}_{p'}(G)\). This subgroup is
characteristic in \(G\) and every block of \(H\) has defect zero and
is therefore nilpotent. The result just proved gives
\(H\leq\mathbf{Z}(G)\mathbf{O}_p(G)\). Hence
\(\mathbf{O}_{p'}(G)=H\leq\mathbf{Z}(G)\).

\medskip\noindent
\textbf{Step 2.} Set \(Q = \tilde{P} \cap \mathbf{F}^{*}(G)\). Then \(\tilde{G} = C_{\tilde{G}}(Q)\mathbf{F}^{*}(G)\) and \(G = C_{G}(Q)\mathbf{F}^{*}(G)\).

\smallskip

Set \(\tilde{K} = C_{\tilde{G}}(Q)\mathbf{F}^*(G)\) and \(K = C_G(Q)\mathbf{F}^*(G)\).
Since \(\tilde{P}\) is abelian and \(\tilde{G} = \tilde{P}G\),
we have \(\tilde{K} = \tilde{P}K\). By the first and second paragraph of the proof,
the block \(b_{\mathbf{F}^*(G)}\) has defect group $Q$ and it is \(\tilde{G}\)-invariant.
The same argument {as in the proof of} Lemma~\ref{Normal-P-cap-N} proves that \(\tilde{K}\) is normal in \(\tilde{G}\).
Since \(b\) covers \(b_K\), the block \(b\) of $\tilde G$ covers the block \(b_K\) of $\tilde K$. Therefore \(\tilde{P}\) is a defect group of the block \(b_K\) of $\tilde K$. Since \(\IBr(b)\) is a single \(\Aut(G)_{b}\)-orbit, by Lemma \ref{Fong} and \ref{Normal-P-cap-N}, \(\IBr(b_K)\) is a single \(\Aut(K)_{b_K}\)-orbit.

Since \( C_{\tilde{G}}(\tilde{P}) \leq C_{\tilde{G}}(Q) \leq \tilde{K} \), we may choose a common maximal Brauer pair \( (\tilde{P}, e_{\tilde{P}}) \) for the block \( b \) of $\tilde G$ and the block \( b_{K} \) of $\tilde K$. Since \(K\) is \(\Aut(G)_b\)-invariant (see Lemma \ref{Normal-P-cap-N}), by the second paragraph of the proof, the block \(b_K \) of $K$ is \(G\)-invariant and thus the block \(b_K \) of $\tilde K$ is \(\tilde G\)-invariant. The Frattini argument then gives
\(
\tilde{G} = \tilde{K} \cdot N_{\tilde{G}}(\tilde{P}, e_{\tilde{P}}).
\)
Since the quotient group \( N_{\tilde{G}}(\tilde{P}, e_{\tilde{P}})/C_{\tilde{G}}(\tilde{P}) \) is a \( p' \)-group, \( \tilde{G}/\tilde{K} \) is a \( p' \)-group. By \cite[Corollary]{Zhou}, the block \( b_{K} \) of $\tilde{K}$ is not inertial. Consequently, ($\tilde K$, $K$, $b_K$) is a counterexample of Theorem~\ref{Main}.

However, if $K<G$, since \(\mathbf{Z}(G) \leq \mathbf{Z}(K)\) and thus $\mathbf{O}_{p'}(\mathbf{Z}(G))\leq \mathbf{O}_{p'}(\mathbf{Z}(K))$, we have
\(|K : \mathbf{O}_{p'}(\mathbf{Z}(K))| < |G : \mathbf{O}_{p'}(\mathbf{Z}(G))|\). This contradicts the minimality of the index $|G : \mathbf{O}_{p'}(\mathbf{Z}(G))|$. So $K=G$ and hence $\tilde K=\tilde G$.

\medskip\noindent
\textbf{Step 3.} The layer \(\mathbf{E}(G)\) is nontrivial.

\smallskip
Suppose \(\mathbf{E}(G) = 1\). Then we have \(\mathbf{F}^*(G) = \mathbf{F}(G)\). By Step 1, we have \(\mathbf{O}_{p'}(G) \leq \mathbf{Z}(G)\), so that \(\mathbf{F}^*(G) = \mathbf{O}_p(G) \mathbf{Z}(G)\). Since \(P\) is abelian, {$P$ centralizes $\mathbf{F}^*(G)$}; the inclusion \(C_G(\mathbf{F}^*(G)) \subseteq \mathbf{F}^*(G)\) therefore gives \(P \subseteq \mathbf{F}^*(G)\).
Consequently we have
\(
P = \mathbf{O}_p(G) = Q\), $\mathbf{F}^{*}(G)\leq C_G(Q)$ and therefore \(G= C_G(Q).
\)
Thus the block \(b\) of $G$ is nilpotent and so is the block $b$ of $\tilde G$. This contradicts the earlier assumption on $b$.

\medskip\noindent
\textbf{Step 4.} Let \(\{X_i\}_{i=1}^r\) be the components of \(G\). Then each \(X_i\) is normal in \(\tilde{G}\).

\smallskip
For each component \(X_i\), let \(d_i\) be the block of \(X_i\)
covered by \(b_{\mathbf{E}(G)}\).
We first claim that none of the blocks \(d_i\) is nilpotent. Set
$
 I=\{i\mid 1\leq i\leq r, d_i\, \text{is nilpotent}\}$.
Every element of \(\Aut(G)_b\) stabilizes the characteristic
subgroup \(\mathbf{E}(G)\) and its unique covered block
\(b_{\mathbf{E}(G)}\), and permutes the components \(X_i\) and their
blocks \(d_i\). Since nilpotency is preserved by automorphisms, the set
\(I\) is invariant under this permutation action. Hence, if \(I\neq
\varnothing\), then
\[
 L=\prod_{i\in I}X_i
\]
is an \(\Aut(G)_b\)-invariant normal subgroup of \(G\). By
\cite[Lemma~2.7(3)]{Ar}, the block \(b_L\) of \(L\) covered by \(b\)
is nilpotent. Step~1 therefore yields
$L\leq \mathbf{Z}(G)\mathbf{O}_p(G)$. This is impossible because $L$ is a
nonempty central product of components. Consequently,
\(I=\varnothing\) and every \(d_i\) is nonnilpotent.

Fix \(i\) and set \(D_i=Q\cap X_i\). By
\cite[Theorem~9.26]{N1}, \(D_i\) is a defect group of \(d_i\). Since \(d_i\) is not
nilpotent, we have
\[
 Q\cap\mathbf{Z}(X_i)<D_i=Q\cap X_i.
\]
Let \(x\in C_{\tilde G}(Q)\). Since \(x\) centralizes \(Q\), we have
\[
 D_i=(Q\cap X_i)^x=Q\cap X_i^x.
\]
In particular, \(D_i\leq X_i\cap X_i^x\). If \(X_i^x\neq X_i\),
then \(X_i\) and \(X_i^x\) are distinct components and hence commute.
It follows that
\[
 X_i\cap X_i^x\leq\mathbf{Z}(X_i),
\]
and therefore \(D_i\leq\mathbf{Z}(X_i)\), a contradiction. Thus every
element of \(C_{\tilde G}(Q)\) normalizes \(X_i\). Finally, since Step 2 gives $\tilde G=C_{\tilde G}(Q)\mathbf{F}^*(G)$ and the subgroup
\(\mathbf{F}^*(G)\) also normalizes every component, \(X_i\trianglelefteq\tilde G\) for every \(i\), as required.

\medskip\noindent
\textbf{Step 5.} \(\tilde{P}\mathbf{E}(G)/\mathbf{E}(G)\) is a Sylow \(p\)-subgroup of \(\tilde{G}/\mathbf{E}(G)\) and
\(
\tilde{G}/\mathbf{E}(G) = \mathbf{O}_{p',\,p,\,p'}\bigl( \tilde{G}/\mathbf{E}(G) \bigr).
\)

\smallskip
By {Steps~1 and~2}, we have
$
 G=C_G(Q)\mathbf{F}^*(G)$ and
 $\mathbf{O}_{p'}(G)\leq\mathbf{Z}(G)$.
Moreover, we have \[\mathbf{O}_p(G)=P\cap \mathbf{F}(G)=\tilde P\cap \mathbf{F}(G)=\tilde P\cap {\mathbf{F}^{*}(G)}\cap \mathbf{F}(G)=Q\cap\mathbf{F}(G).\] Since \(P\) is abelian,
\(\mathbf{O}_p(G)\) centralizes \(\mathbf{F}(G)\) and hence
\(\mathbf{E}(G)\) because
\([\mathbf{F}(G),\mathbf{E}(G)]=1\). Consequently \(\mathbf{O}_p(G)\) centralizes
\(\mathbf{F}^*(G)\). Since \(\mathbf{O}_p(G)\) is contained in \(Q\), it is also
centralized by \(C_G(Q)\). Therefore we have
$\mathbf{O}_p(G)\leq\mathbf{Z}(G)$.
Together with Step~1, this implies
\(\mathbf{F}(G)=\mathbf{Z}(G)\) and hence
$\mathbf{F}^*(G)=\mathbf{Z}(G)\mathbf{E}(G)$.

By Step~4, every component of \(G\) is normal. Lemma~\ref{Solv} now
shows that \(G/\mathbf{F}^*(G)\) is solvable. Since
\(\mathbf{F}^*(G)/\mathbf{E}(G)\) is abelian, \(G/\mathbf{E}(G)\) is
solvable, hence \(\tilde G/\mathbf{E}(G)\) is solvable because
\(\tilde G/G\) is a \(p\)-group. For every characteristic subgroup
\(L\) of \(G\) containing \(\mathbf{E}(G)\), the block of \(L\)
covered by \(b\) is \(G\)-invariant by the third paragraph of the
proof. Lemma~\ref{Abel} applied with \(N=\mathbf{E}(G)\) therefore
gives $P\mathbf{E}(G)/\mathbf{E}(G)\in\operatorname{Syl}_p(G/\mathbf{E}(G))$.
Since \(\tilde G=\tilde P G\) and \(\tilde P\cap G=P\), we have $
 \tilde P\mathbf{E}(G)/\mathbf{E}(G)\in
 \operatorname{Syl}_p(\tilde G/\mathbf{E}(G))$. Hence \cite[Lemma~6.3.3]{Go} applied
to the solvable group \(\tilde G/\mathbf{E}(G)\) yields
$\tilde G/\mathbf{E}(G)=
 \mathbf{O}_{p',p,p'}(\tilde G/\mathbf{E}(G)).$

\medskip

Let \(\tilde T\) and \(\tilde M\) be the inverse images in \(\tilde
G\) of \(\mathbf{O}_{p',p}(\tilde G/\mathbf{E}(G))\) and
\(\mathbf{O}_{p'}(\tilde G/\mathbf{E}(G))\) respectively. The Sylow
statement just proved gives
\[
 \tilde T=\tilde P\tilde M,
 \qquad \tilde M/\mathbf{E}(G)\text{ a }p'\text{-group},
 \qquad \tilde G/\tilde T\text{ a }p'\text{-group}.
\]
By the second paragraph of the proof, \(b_{\mathbf{E}(G)}\) is $\tilde P$-invariant and hence it is also a block
of \(\tilde P\mathbf{E}(G)\) with defect group \(\tilde P\).
Since distinct components {commute}, $d_i$ is the unique block of $X_i$ covered by $b_{\mathbf{E}(G)}$. Consequently, \(d_i\) is
\(\tilde P\)-invariant and hence it is also a block of \(\tilde P X_i\) with
defect group \(\tilde P\).
Lemma~\ref{Normal-one-orbit2} gives that \(\IBr(d_i)\) is a
single \(\Aut(X_i)_{d_i}\)-orbit. Since \(X_i\) is
quasisimple, the hypothesis of Proposition~\ref{Reduction} applied to
\(X_i\trianglelefteq\tilde P X_i\) shows that \(d_i^{\tilde P X_i}\) is inertial. {By} \cite[Proposition~2.3]{ZZ1},
\((b_{\mathbf{E}(G)})^{\tilde P\mathbf{E}(G)}\) is inertial.

Let \(b_{\tilde T}\) be the block of \(\tilde T\) covered by \(b\).
Since \(\tilde T=\tilde P\tilde M\) and
\(\tilde M/\mathbf{E}(G)\) is a \(p'\)-group,
\cite[Corollary~1.4]{ZZ} implies that \(b_{\tilde T}\) is inertial.
Finally, since  \(\tilde G/\tilde T\) is a \(p'\)-group,
\cite[Corollary]{Zhou} implies that \(b^{\tilde G}\) is inertial. This contradicts the assumption and completes
the proof.
\end{proof}

\section{Blocks of quasisimple groups}\label{Sec:3}

Let \(\mathbf{G}\) be a connected reductive algebraic group defined over an algebraically closed field \(\overline{\mathbb{F}}_\ell\), where $\ell$ is a prime. Let \(q\) be a power of \(\ell\), let \(F: \mathbf{G} \to \mathbf{G}\) be a Frobenius endomorphism endowing \(\mathbf{G}\) with an \(\mathbb{F}_q\)-structure{, and let} \((\mathbf{G}^*, F)\) be {dual to} $({\mathbf G}, F)$ (see \cite[Page 118]{CE}). Set \(G = \mathbf{G}^F\) and \(G^* = {\mathbf{G}^{*F}}\).
The set \(\Irr(G)\) of irreducible ordinary characters of \(G\) {decomposes into the disjoint union} of Lusztig series
\[
\Irr(G) = \coprod_{(s)} \mathcal{E}(G, s),
\]
where \((s)\) runs over conjugacy classes of semisimple elements \(s \in G^*\). The characters in the series \(\mathcal{E}(G, 1)\) are called {unipotent characters}. For any block $b$ of $G$, there is a semisimple element $s$ such that \(\Irr(b) \cap \mathcal{E}(G, s) \neq \emptyset\), where $\Irr(b)$ denotes the set of all irreducible ordinary characters in the block $b$.
The block \(b\) of \(G\) is {unipotent} if \(s=1\). In what follows, we assume throughout that \(p \neq \ell\).

\begin{lem}\label{unipotentcentral} Assume that \(\mathbf{G}\) is semisimple and that $p$ divides \(|\mathbf{Z}(\mathbf{G})^F|\).  Then \(G=\mathbf{G}^F\) has no unipotent block with central defect group.
\end{lem}

\begin{proof}
Suppose that $b$ is a unipotent block of $G$ with central defect group $D$.
Let $\pi:\mathbf G_{\mathrm{sc}}\to\mathbf G$ be the simply connected
covering. By \cite[Proposition~22.7]{MT}, $F$ lifts to
$\mathbf G_{\mathrm{sc}}$. Set $G_{\mathrm{sc}}=\mathbf G_{\mathrm{sc}}^F$
and $H=\pi(G_{\mathrm{sc}})\trianglelefteq G$.
Choose a unipotent character $\chi\in\Irr(b)$. By
\cite[Proposition~15.9]{CE}, $\chi\circ\pi$ is an irreducible unipotent
character of $G_{\mathrm{sc}}$. Let $b_{\mathrm{sc}}$ be the block containing $\chi\circ\pi$. Note that $\chi\circ\pi$ is the inflation of $\chi_H$ along the homomorphism $\pi: G_{\mathrm{sc}}\rightarrow H$ and hence that $\chi_H$ is irreducible.
The block of $H$ containing $\chi_H$ is covered by $b$ and dominated
by $b_{\mathrm{sc}}$. By
\cite[Theorems~9.9(c), 9.10 and~9.26]{N1}, we may choose a defect group
$D_{\mathrm{sc}}$ of $b_{\mathrm{sc}}$ such that
$\pi(D_{\mathrm{sc}})=D\cap H$.
By \cite[Proposition~3.6.8]{C}, $D\leq\mathbf Z(\mathbf G)^F$.
Since $\pi^{-1}(\mathbf Z(\mathbf G))=\mathbf Z(\mathbf G_{\mathrm{sc}})$,
it follows that $D_{\mathrm{sc}}$ is central.

We also have $p\mid|\mathbf Z(\mathbf G_{\mathrm{sc}})^F|$.
Indeed, $\pi$ induces an $F$-equivariant epimorphism
$$A=\mathbf Z(\mathbf G_{\mathrm{sc}})_p\to
B=\mathbf Z(\mathbf G)_p$$ between finite abelian $p$-groups.
If $A^F=1$, then $F-1$ is bijective on $A$, on
the invariant kernel of $A\to B$ and hence on $B$.
This contradicts $B^F\neq1$.

Now $\mathbf G_{\mathrm{sc}}$ is a direct product of simply connected
simple groups. Taking representatives $\mathbf S_1,\ldots,\mathbf S_t$
of the $F$-orbits on its components gives
(see \cite[Corollary~1.5.16]{GM})
\[
 G_{\mathrm{sc}}\cong\prod_{i=1}^t S_i,
 \qquad S_i=\mathbf S_i^{F^{m_i}},
\]
where $m_i$ is the size of the orbit of $\mathbf S_i$.
Under this identification, write $b_{\mathrm{sc}}=b_1\otimes\cdots\otimes b_t$.
Each $b_i$ is unipotent with central defect group. Since
$\mathbf Z(\mathbf G_{\mathrm{sc}})^F\cong
\prod_i\mathbf Z(\mathbf S_i)^{F^{m_i}}$, at least one factor has center
of order divisible by $p$. Replacing $\mathbf G$, $F$ and $b$ by that
factor, its corresponding Frobenius endomorphism and its block, we may
therefore assume that $\mathbf G$ is simple and simply connected.

Suppose first that $\mathbf{G}$ is of classical type. We claim that $b$ is the principal block. If $p=2$, this follows from \cite[Theorem 21.14]{CE}. If $p$ is odd, since $p\mid|\mathbf{Z}(\mathbf{G})^F|$, inspection of \(|\mathbf{Z}(\mathbf{G})^F|\) (see \cite[Table~24.2]{MT}) shows that $\mathbf{G}$ must be of type $A$. Then by \cite[Theorem 22.9 and Lemma 22.17]{CE}, the block $b$ is principal. {This proves the claim.} Consequently $G$ has {a} nontrivial central Sylow $p$-subgroup and is not perfect.
By \cite[Theorem 24.17]{MT}, $G$ is one of the following exceptions
$$\mathrm{SL}_2(2),~\mathrm{SL}_2(3),~\mathrm{SU}_3(2),~\mathrm{Sp}_4(2), ~\mathrm{G}_2(2).$$
However, for every group {in this list}, either its center is trivial or its Sylow $p$-subgroups are not contained in the center, yielding a contradiction.
	
It remains to consider the exceptional types. Since
$p\mid|\mathbf{Z}(\mathbf{G})^F|$, inspection of the centers
(see \cite[Table~24.2]{MT}) leaves precisely the following three
possibilities. Writing
$\Phi_d=\Phi_d(q)$, the standard order formulas and
\cite[Lemma~5.2]{M2} give
\[
\begin{array}{c|c|c}
G & |{\bf Z}(G)| & |G|_p \\
\hline
E_{6,\mathrm{sc}}(q),\ 3\mid(q-1)
  & 3 & (\Phi_1^6\Phi_3^3\Phi_9)_3 \\
{}^2E_{6,\mathrm{sc}}(q),\ 3\mid(q+1)
  & 3 & (\Phi_2^6\Phi_6^3\Phi_{18})_3 \\
E_{7,\mathrm{sc}}(q)
  & 2 & (\Phi_1^7\Phi_2^7\Phi_4^2\Phi_8)_2
\end{array}
\]
where $p=3$ in the first two rows and $p=2$ in the last row. Comparing
these $p$-parts with the unipotent character degree formulas in
\cite[Page~480-483]{C} uniformly gives
\(
 \chi(1)_p<\frac{|G|_p}{p}
\)
for every unipotent character $\chi$ of $G$.

Suppose that $\chi$ belongs to a block with central defect
group $D$. Then the definition of the height $h(\chi)$ gives
\(
 \chi(1)_p=\frac{|G|_p}{|D|}p^{h(\chi)}.
\)
Since $D\leq {\bf Z}(G)$ and $|{\bf Z}(G)|=p$ in all three cases, the right-hand
side is at least $|G|_p/p$, a contradiction. Thus no unipotent
character in any of the remaining cases lies in a block with central
defect group, and the proof is complete.
\end{proof}

\begin{lem}\label{unipotentdegree}
Let \(b\) be a unipotent block of \(G = \mathbf{G}^F\). Then either \(b\) has {a} central defect group or \(b\) contains at least two unipotent characters whose degrees have distinct \(\ell\)-parts.
\end{lem}

\begin{proof} Let $\mathbf G_{\mathrm{sc}}$ be the simply connected
covering of $[\mathbf G,\mathbf G]$ and let
$\pi:\mathbf G_{\mathrm{sc}}\to\mathbf G$ be the resulting central
isotypy (see \cite[Remark~1.5.13]{GM}).
The Frobenius endomorphism \( F \) on \( \mathbf{G} \) lifts to a Frobenius endomorphism on \( \mathbf{G}_{\rm sc}\), still denoted by $F$. Set $G_{\rm sc}=\mathbf{G}_{\rm sc}^F$.
By \cite[Proposition 15.9]{CE},
the restriction of characters induces a bijective correspondence between the unipotent characters of $G_{\rm sc}$ and those of $G$.
This correspondence preserves character degrees.
Moreover, by \cite[Theorem 17.1]{CE}, the bijective correspondence is compatible with the block decomposition of characters. Let $b_{\rm sc}$ be the block of $G_{\rm sc}$ corresponding to the block $b$. The block $b_{\rm sc}$ is unipotent.

Suppose first that \(b_{\mathrm{sc}}\) does not have a central defect group.
Since $\mathbf G_{\mathrm{sc}}$ is semisimple, quasi-central
defect and central defect of characters coincide.
By \cite[Lemma 3.15]{MNST}, there exist two unipotent characters in \(\Irr(b_{\mathrm{sc}})\) whose degrees have distinct \(\ell\)-parts. The unipotent character bijection above then implies that \(\Irr(b)\) also contains such a pair.

Suppose that $b_{\mathrm{sc}}$ has central defect group.
Lemma~\ref{unipotentcentral} and \cite[Proposition~3.6.8]{C} imply
$p\nmid|{\bf Z}(G_{\mathrm{sc}})|$, so $b_{\mathrm{sc}}$ has defect zero.
Choose a unipotent character $\chi\in\Irr(b)$ and let $Z_0$ be the
Sylow $p$-subgroup of $(\mathbf Z(\mathbf G)^{\circ})^F$.
By \cite[1.5.11]{GM}, we have
\(
 \chi(1)_p=|G_{\mathrm{sc}}|_p=|G:Z_0|_p.
\)
Since $Z_0\leq {\bf Z}(G)$, \cite[Theorem~9.13]{N1} shows that $Z_0$ is a
defect group of $b$, as required.
\end{proof}

For a semisimple \(p'\)-element \(s \in G^*\), set \(\mathcal{E}_p(G,s) = \bigcup_t \mathcal{E}(G,st)\), where \(t\) runs over the \(p\)-elements of \(G^*\) commuting with \(s\).  Then \(\cE_p(G,s)\) is a union of sets of
irreducible complex characters of blocks of \(G\). For any block \(b\) of \(G\) with \(\Irr(b) \subseteq \mathcal{E}_p(G,s)\), we have \(\Irr(b) \cap \mathcal{E}(G,s) \neq \emptyset\) (see \cite[Theorem 9.12]{CE}). In this case, \(b\) is said to lie in \(\mathcal{E}_p(G,s)\).

From now on, we assume that \(\mathbf{G}\) is simple of simply connected type. Let \(\iota: \mathbf{G} \hookrightarrow \tilde{\mathbf{G}}\) be a regular embedding with dual map \(\iota^*: \tilde{\mathbf{G}}^* \to \mathbf{G}^*\) (see \cite[15.1]{CE}). Identifying \(\mathbf{G}\) with \(\iota(\mathbf{G})\), we have \(\tilde{\mathbf{G}} = \mathbf{G} \mathbf{Z}(\tilde{\mathbf{G}})\). The Frobenius endomorphism \(F\) on \(\mathbf{G}\) {extends to a Frobenius endomorphism of} \(\tilde{\mathbf{G}}\), still denoted by \(F\). Set \(\tilde{G}= \tilde{\mathbf{G}}^F\) and \(\tilde{G}^*= \tilde{\mathbf{G}}^{*F}\).
We record \cite[Theorem~1]{KM2} in the notation used here.

\begin{lem}\label{Jordan}
Let \(\mathbf X\) be a connected reductive group in characteristic
\(\ell\) with connected center and with simple, simply connected derived
subgroup, and let \(F:\mathbf X\to\mathbf X\) be a Frobenius
endomorphism. Let \(\mathbf H\) be an \(F\)-stable Levi subgroup of
\(\mathbf X\) and \(s\in\mathbf H^{*F}\)
be a semisimple \(p'\)-element and
set \(H=\mathbf H^F\). Then, for every \(p\)-element
\(t\in C_{\mathbf H^*}(s)^F\), there exists a map
\[
 {J}_{t}^{\mathbf H}:
 \bigl\{\text{unipotent \(p\)-blocks of }
 C_{\mathbf H^*}(st)^F\bigr\}
 \longrightarrow
 \bigl\{\text{\(t\)-twin \(p\)-blocks in }\mathcal E_p(H,s)\bigr\}
\]
such that if \(d\) is a unipotent \(p\)-block of
\(C_{\mathbf H^*}(st)^F\) and
\(
 \eta\in
 \mathcal E\bigl(C_{\mathbf H^*}(st)^F,1\bigr)\cap\Irr(d),
\)
then the character in \(\mathcal E(H,st)\) corresponding to \(\eta\)
under Jordan decomposition belongs to \({J}_{t}^{\mathbf H}(d)\).
\end{lem}

Here \(t\)-twin blocks are defined in
\cite[Section~4.2]{KM2}.


\begin{lem}\label{Twin}
Assume that \(\mathbf G\) has type \(E_8\) and \(p\geq7\).
Let \(s\in G^*\) be a semisimple \(p'\)-element, let \(b\) be a
block lying in \(\mathcal E_p(G,s)\), and set
\(\mathbf C=C_{\mathbf G^*}(s)\).
If \(\IBr(b)\) is a single \(\Aut(G)_b\)-orbit, then no
non-trivial \(1\)-twin block contains \(b\).
In particular, let \(f\) be a unipotent block of \(\mathbf C^F\).
If the Jordan correspondent of some unipotent character in \(f\)
lies in \(\Irr(b)\cap\mathcal E(G,s)\), then the Jordan
correspondents of all unipotent characters in \(f\) lie in this set.
\end{lem}

\begin{proof}
Since \(p\) is good for \(\mathbf G\), by \cite[Theorem~14.4]{CE} and the proof of
\cite[Lemma~3.5]{MNS}, the set
\(\Irr(b)\cap\mathcal E(G,s)\) is an \(\Aut(G)_b\)-stable basic
set.  Hence \cite[Lemma~3.2]{MNS} implies that
all its characters in $\Irr(b)\cap\mathcal E(G,s)$ have the same degree.

Suppose that \(b\) belongs to a non-trivial \(1\)-twin block.
Then \(p\mid(q+1)\) and the relevant cuspidal pairs are those in
\cite[Table~1]{KM2}. If \(s=1\), then \(b\) is unipotent.
By the definition of a non-trivial \(1\)-twin in
\cite[Section~4.2]{KM2}, we may choose a unipotent
\(2\)-cuspidal pair \((\mathbf M^*,\lambda^*)\) from \cite[Table~1]{KM2}
such that the block it labels has the image \(b\) under
the map of \cite[Proposition~4.1 (b)]{KM2}.
Let \((\mathbf L,\mu)\) be the pair corresponding to
\((\mathbf M^*,\lambda^*)\) under the bijection in \cite[Proposition~2.2]{KM2}.
Since \(t=1\), the relationship \(\mathord{\to}_1\) in
\cite[Definition 3.7]{KM2} is the identity map. Hence the
construction in Proposition~4.1 shows that
the image of \((\mathbf L, \mu)\) under the bijection in \cite[Theorem~22.9(ii)]{CE} is equal to \(b\).
Moreover, since \(s=1\), the Levi subgroup \(\mathbf M^*\) is
\(2\)-split in \(\mathbf G^*\), so we have
\(C_{\mathbf G^*}({\bf Z}^\circ(\mathbf M^*)_2)=\mathbf M^*\).
Thus \(\mathbf L\) is dual to \(\mathbf M^*\) (see \cite[Remark~4.2]{KM2}).
Under the non-trivial twin assumption, \cite[Table~1]{KM2} gives the rational
type \(\Phi_2^2.{}^2E_6(q)\) or \(\Phi_2.E_7(q)\) for
\(\mathbf M^{*F}\); see also \cite[Lemma~6.10(b),(c)]{KM2}.
Here the factors \(\Phi_2^2\) and \(\Phi_2\) record the
polynomial order of \({\bf Z}^\circ(\mathbf M^*)\), respectively
(see \cite[Theorem~13.1 and Proposition~13.2(i)]{CE}).
So we have
\(
 |{\bf Z}^\circ(\mathbf L)^F|
   =\Phi_2(q)^a=(q+1)^a>1,\, a\in\{1,2\}.
\)
Since \(p\mid(q+1)\), we also have
\(|{\bf Z}^\circ(\mathbf L)^F|_p>1\).
By \cite[Theorem~22.9(ii)]{CE}, a defect group of \(b\) contains
a Sylow \(p\)-subgroup of \({\bf Z}^\circ(\mathbf L)^F\), and
is therefore non-trivial. Since
\({\bf Z}(G)=1\), this defect group is not central.  Lemma~\ref{unipotentdegree}
therefore contradicts the equal degrees of the unipotent characters in \(b\).

Assume \(s\ne1\).
Clearly $s$ is not central in $\mathbf G^*$. Since \(\mathbf G^*\) is simply connected,  by \cite[Theorem~13.14]{CE}, \(\mathbf C\) is
connected.
By \cite[Table~1]{KM2}, \(\mathbf C^F\) has rational type
\({}^2E_6(q).{}^2A_2(q)\) or \(E_7(q).A_1(q)\).
$\mathbf C$ is semisimple of rank~8,
whereas a proper Levi subgroup of \(\mathbf G^*\) has semisimple
rank at most~7. Thus \(\mathbf C\) is not contained in a proper
Levi subgroup and hence \(s\) is isolated.
The simply connected covering of \(\mathbf C\) has center of
order dividing \(4\) or \(9\), which is surjectively mapped onto \({\bf Z}(\mathbf C)\); see \cite[Section~1.11]{C}.
Moreover, \cite[Proposition~3.6.8]{C} gives
\({\bf Z}(\mathbf C^F)={\bf Z}(\mathbf C)^F\).
Consequently, \(|{\bf Z}(\mathbf C^F)|\) divides \(4\) or \(9\) and
and hence \(p\nmid|{\bf Z}(\mathbf C^F)|\).

We choose a unipotent \(2\)-cuspidal pair
\((\mathbf M_d,\lambda_d)\) from \cite[Table~1]{KM2}, which labels a block
\(d\) of \(\mathbf C^F\), which has the image \(b\) under the map in
\cite[Proposition~4.1(b)]{KM2}. By
\cite[Table~1]{KM2}, \(\mathbf M_d^F\) has rational type
\(\Phi_2^2.{}^2E_6(q)\) or \(\Phi_2.E_7(q)\).
The polynomial order formula used above therefore yields
\(
 |{\bf Z}^\circ(\mathbf M_d)^F|_p
   =(q+1)_p^a>1, a\in\{1,2\}
\)
since \(p\mid(q+1)\).
By \cite[Theorem~22.9(ii)]{CE}, a defect group of \(d\) may be
chosen to contain a Sylow \(p\)-subgroup of
\({\bf Z}^\circ(\mathbf M_d)^F\). It is therefore non-trivial
and non-central since \(p\nmid|{\bf Z}(\mathbf C^F)|\).
By \cite[Theorem~22.9(i)]{CE}, the unipotent characters of \(d\)
form the single \(2\)-Harish--Chandra series labelled by
\((\mathbf M_d,\lambda_d)\).
Lemma~\ref{unipotentdegree} gives \(\eta_1,\eta_2\) in this series
with different \(\ell\)-parts of their degrees. Since $s$ is not central and isolated,
by \cite[Proposition~4.10]{KM2}, the bijection from the series
defining \(b\) to this series agrees in degrees with Jordan
decomposition.  The Jordan degree formula therefore gives
\(\chi_1,\chi_2\in\Irr(b)\cap\mathcal E(G,s)\) with
\(
 \chi_1(1)_\ell=\eta_1(1)_\ell\ne\eta_2(1)_\ell=\chi_2(1)_\ell,
\)
again a contradiction.  Thus no non-trivial \(1\)-twin block contains \(b\).

Finally, under the hypothesis on \(f\), Lemma~\ref{Jordan} shows
that \(J_1^{\mathbf G}(f)\) contains \(b\), so the first assertion
gives \(J_1^{\mathbf G}(f)=b\).  Another application of
Lemma~\ref{Jordan} proves the last assertion.
\end{proof}

\begin{lem}\label{Good}
Let \(s \in G^*\) and \(\tilde{s} \in \tilde{G}^{*}\) be semisimple \(p'\)-elements such that \(\iota^\ast(\tilde{s}) = s\).
Let $b$ and $\tilde b$ be respective blocks of $G$ and $\tilde G$ lying in $\mathcal{E}_p(G,s)$ and $\mathcal{E}_p(\tilde{G},\tilde{s})$
such that \(\tilde{b}\) covers \(b\).
Let \(\mathbf{L}\) be a Levi subgroup of \(\mathbf{G}\) such that \(\mathbf{L}^* \geq C_{\mathbf{G}^*}(s)\).
If \(p\) is good for \(\mathbf{L}\) and \(\IBr(b)\) is {\it a single \(\Aut(G)_b\)-orbit},
then \(\tilde{b}\) is nilpotent with abelian defect groups.
\end{lem}

\begin{proof}
Since \(p\) is good for \(\mathbf{L}\), by
\cite[Theorem 14.6]{CE}, the restrictions of elements of \(\mathcal{E}(\mathbf{L}^F,s)\) to $\mathbf{L}_{p'}^F$ are linearly independent. By \cite[Theorem 10.1]{CE}, Lusztig induction $R_{\mathbf{L}}^{\mathbf{G}}$
induces a Morita equivalence between the $p$-blocks in $\mathcal{E}_p(\mathbf{L}^F,s)$ and $\mathcal{E}_p(G,s)$, which sends $\mathcal{E}(\mathbf{L}^F,s)$ to $\mathcal{E}(G,s)$. Thus the restrictions of elements of \(\mathcal{E}(G,s)\) to $G_{p'}$ are linearly independent.
Since \(\IBr(b)\) is a single \(\Aut(G)_b\)-orbit, by \cite[Lemma 3.2]{MNS}, \(\Irr(b) \cap \mathcal{E}(G,s)\) is a single \(\Aut(G)_b\)-orbit.
By \cite[Theorem 2.5.1]{GLS}, \(\Aut(G)\) is induced by \(\tilde{G} \rtimes \mathcal{D}\),
where \(\mathcal{D}\) denotes the group generated by suitable graph and field automorphisms of \(G\).
{Consequently}, the stabilizers \(\tilde{G}_\chi\) for \(\chi \in \Irr(b) \cap \mathcal{E}(G,s)\) are \(\tilde{G} \rtimes \mathcal{D}\)-conjugate.
By \cite[Proposition 15.6 and Theorem 15.11\('\)]{CE}, all characters in \(\Irr(\tilde{b}) \cap \mathcal{E}(\tilde{G},\tilde{s})\) have equal degree.

Set \(\tilde{\mathbf{L}} = \mathbf{L} \mathbf{Z}(\tilde{\mathbf{G}})\). Then $\tilde{\mathbf{L}}$ is a Levi subgroup of $\tilde{\mathbf{G}}$, the center of \(\tilde{\mathbf{L}}\) is connected and $C_{\tilde{\mathbf{G}}^\ast}(\tilde{s})$ is contained in $\tilde{\mathbf{L}}^\ast$.
Set
\(\mathbf C=C_{\tilde{\mathbf G}^{*}}(\tilde s)\).
Since the center of \(\tilde{\mathbf G}\) is connected, by
\cite[Theorem~4.5.9]{C}, \(\mathbf C\) is connected.
Let \(\tilde{b}_{\tilde{\mathbf{L}}^F}\) be the block of \(\tilde{\mathbf{L}}^F\) corresponding to \(\tilde{b}\) via the Bonnaf\'e-Rouquier correspondence (see \cite[Theorem 10.1]{CE}).
By \cite[Proposition 3.3.18]{GM}, all characters in \(\Irr(\tilde{b}_{\tilde{\mathbf{L}}^F}) \cap \mathcal{E}(\tilde{\mathbf{L}}^F, \tilde{s})\) have equal degree.
Since \(p\) is good for \(\mathbf{L}\), by
\cite[Theorem 14.4]{CE}, \(\Irr(\tilde{b}_{\tilde{\mathbf{L}}^F}) \cap \mathcal{E}(\tilde{\mathbf{L}}^F, \tilde{s})\) is a basic set for \(\tilde{b}_{\tilde{\mathbf{L}}^F}\).
However, it follows from \cite[Corollary~3.17]{N1} that $\tilde{b}_{\tilde{\mathbf{L}}^F}$ at least has a height zero irreducible Brauer character. So \(\Irr(\tilde{b}_{\tilde{\mathbf{L}}^F}) \cap \mathcal{E}(\tilde{\mathbf{L}}^F, \tilde{s})\) at least has a height zero character and hence all characters in
\(\Irr(\tilde{b}_{\tilde{\mathbf{L}}^F}) \cap \mathcal{E}(\tilde{\mathbf{L}}^F, \tilde{s})\) have height zero.

We claim that no
non-trivial \(1\)-twin block contains \(\tilde b_{\tilde{\mathbf L}^F}\).  By
\cite[Section~4.2]{KM2}, this is automatic unless
\([\tilde{\mathbf L},\tilde{\mathbf L}]\) has a component of type
\(E_8\).  When \([\tilde{\mathbf L},\tilde{\mathbf L}]\) has a component of type
\(E_8\), \(\mathbf G\) is of type \(E_8\), $\mathbf L=\mathbf G$, $\tilde{\mathbf L}=\tilde{\mathbf G}$ and
$\tilde b_{\tilde{\mathbf L}^F}=\tilde  b$; moreover \(p\geq7\) since \(p\) is good for \(\mathbf L\).  Set
\(\mathbf T=\mathbf Z(\tilde{\mathbf G})\).  Since
\(\mathbf Z(\mathbf G)=1\), we have the \(F\)-stable direct product
\(\tilde{\mathbf G}=\mathbf G\times\mathbf T\) and
\(\tilde b=b\otimes a\) for a block \(a\) of \(\mathbf T^F\).
Lemma~\ref{Twin} applied to \(b\) excludes a non-trivial
\(1\)-twin containing \(b\).  By the direct-product compatibility in
\cite[Proposition~2.4 and Section~4.2]{KM2}, the claim follows for \(\tilde b\).

Let \(\Pi\) be the set of unipotent blocks $f$ of \(\mathbf C^F\) such that the Jordan correspondent of some irreducible unipotent character in the block \(f\) lies in \(\tilde{b}_{\tilde{\mathbf{L}}^F} \cap \cE(\tilde{\mathbf{L}}^F,\tilde s)\).
For \(f\in\Pi\), Lemma~\ref{Jordan} and the claim give
\(J_1^{\tilde{\mathbf L}}(f)=\tilde b_{\tilde{\mathbf L}^F}\).  Hence all unipotent characters
in \(f\) have Jordan correspondents in
\(\Irr(\tilde b_{\tilde{\mathbf L}^F})\cap\mathcal E(\tilde{\mathbf L}^F,\tilde s)\).
Since all characters in \(\Irr(\tilde{b}_{\tilde{\mathbf{L}}^F}) \cap \mathcal{E}(\tilde{\mathbf{L}}^F, \tilde{s})\) have equal degree, \cite[Corollary 2.6.6]{GM} implies that the characters in \(\bigcup_{f \in \Pi} \Irr(f) \cap \mathcal{E}({\mathbf C^F},1)\) have equal degree.
By Lemma \ref{unipotentdegree}, all blocks in \(\Pi\) have central defect group.

Fix a \(p\)-element \(\tilde{t} \in \mathbf C^F\) such that \(\Irr(\tilde{b}_{\tilde{\mathbf{L}}^F}) \cap \mathcal{E}(\tilde{\mathbf{L}}^F, \tilde{s}\tilde{t}) \neq \emptyset\).
Let \(\chi \in \Irr(\tilde{b}_{\tilde{\mathbf{L}}^F}) \cap \mathcal{E}(\tilde{\mathbf{L}}^F, \tilde{s}\tilde{t})\),
and let \(J_{\tilde{t}}(\chi)\) denote its Jordan correspondent in \(\mathcal{E}({C_{\tilde{\mathbf{G}}^*}(\tilde{s}\tilde{t})^F}, 1)\).
Similarly, let \(\varphi \in \Irr(\tilde{b}_{\tilde{\mathbf{L}}^F}) \cap \mathcal{E}(\tilde{\mathbf{L}}^F, \tilde{s})\)
and let \(J_1(\varphi)\) denote its Jordan correspondent in \(\mathcal{E}({\mathbf C^F}, 1)\).
By \cite[Corollary 2.6.6]{GM}, we have
$\chi(1) = J_{\tilde{t}}(\chi)(1) |\tilde{\mathbf{L}}^{*F} : C_{\tilde{\mathbf{G}}^*}(\tilde{s}\tilde{t})^F|_{\ell'}$ and
$\varphi(1) = J_1(\varphi)(1) |\tilde{\mathbf{L}}^{*F} : \mathbf C^F|_{\ell'}$.

Suppose \(\tilde{t} \notin \mathbf{Z}(\mathbf C^F)\).
Recall that {$p\neq\ell$} and that \(J_1(\varphi)\) lies in a block {with a central defect group}.
By \cite[Theorem 5.12]{N2}, we {obtain}
\[
\begin{aligned}
\chi(1)_p &= J_{\tilde{t}}(\chi)(1)_p  |\tilde{\mathbf{L}}^{*F} : C_{\tilde{\mathbf{G}}^*}(\tilde{s}\tilde{t})^F|_{p} \\
&\leq |\tilde{\mathbf{L}}^{*F} : \mathbf{Z}(C_{\tilde{\mathbf{G}}^*}(\tilde{s}\tilde{t})^F)|_{p} \\
&< |\tilde{\mathbf{L}}^{*F} : \mathbf{Z}(\mathbf C^F)|_{p} \\
&= J_1(\varphi)(1)_p  |\tilde{\mathbf{L}}^{*F} : \mathbf C^F|_{p} = \varphi(1)_p.
\end{aligned}
\]
This contradicts the fact that \(\varphi\) has height zero: indeed, if
\(|\tilde{\mathbf L}^{F}|_p=p^a\) and \(|D|=p^d\), then every
\(\theta\in\Irr(\tilde b_{\tilde{\mathbf L}^{F}})\) satisfies
\(\theta(1)_p=p^{a-d+h(\theta)}\).  Thus a height-zero character has the
smallest possible \(p\)-part of its degree within the block, whereas the displayed
inequality gives \(\chi(1)_p<\varphi(1)_p\).

Consequently, we have \(\tilde t\in {\bf Z}(\mathbf C^F)\).  Since \(\mathbf C\) is
connected, \cite[Proposition~3.6.8]{C} yields
\(
 {\bf Z}(\mathbf C^F)={\bf Z}(\mathbf C)^F;
\)
hence \(\tilde t\in {\bf Z}(\mathbf C)^F\).  Since \(\tilde s\) and \(\tilde t\) have
coprime orders, both are powers of \(\tilde s\tilde t\). It follows that we have
\(
 C_{\tilde{\mathbf G}^{*}}(\tilde s\tilde t)
 =C_{\tilde{\mathbf G}^{*}}(\tilde s)\cap
   C_{\tilde{\mathbf G}^{*}}(\tilde t)
 =\mathbf C
\).  In particular, we have
\(C_{\tilde{\mathbf G}^{*}}(\tilde s\tilde t)^F=\mathbf C^F\).
Let \(f_\chi\) be the unipotent block of \(\mathbf C^F\)
containing \(J_{\tilde t}(\chi)\).  Lemmas~\ref{Jordan} and~\ref{zh}
give
\(
 \tilde b_{\tilde{\mathbf L}^F} J_1^{\tilde{\mathbf L}}(f_\chi)
   =\tilde b_{\tilde{\mathbf L}^F} J_{\tilde t}^{\tilde{\mathbf L}}(f_\chi)=\tilde b_{\tilde{\mathbf L}^F}.
\)
The claim above therefore gives \(J_1^{\tilde{\mathbf L}}(f_\chi)=\tilde b_{\tilde{\mathbf L}^F}\).
Since \(f_\chi\) contains a unipotent character, Lemma~\ref{Jordan}
at \(t=1\) now implies \(f_\chi\in\Pi\).  Consequently, we have
\(
 J_{\tilde t}(\chi)\in
 \bigcup_{f\in\Pi}(\Irr(f)\cap\mathcal E(\mathbf C^F,1)).
\)

Now let \(\theta\in\Irr(\tilde b_{\tilde{\mathbf L}^{F}})\).  Since
\(\Irr(\tilde b_{\tilde{\mathbf L}^{F}})\subseteq
\mathcal E_p(\tilde{\mathbf L}^{F},\tilde s)\), there is a \(p\)-element
\(\tilde t\in\mathbf C^F\) such that
\(\theta\in\mathcal E(\tilde{\mathbf L}^{F},\tilde s\tilde t)\).  The preceding
argument applies to this \(\tilde t\): its Jordan correspondent belongs to the union
over \(f\in\Pi\) considered above and
\(C_{\tilde{\mathbf G}^{*}}(\tilde s\tilde t)^F=\mathbf C^F\).  The Jordan degree
formula and the equal degree conclusion for that union therefore show that
\(\theta(1)\) is independent of
\(\theta\).
Thus all ordinary irreducible characters of
\(\tilde b_{\tilde{\mathbf L}^{F}}\) have the same degree and, since every block
contains a height-zero ordinary character, they all have height zero.  The Brauer
height-zero theorem \cite{MNST} now implies that
\(\tilde b_{\tilde{\mathbf L}^{F}}\) has an abelian defect group.
By \cite[Theorem 4.1]{MN} and \cite[Theorem 10.1]{CE}, \(\tilde{b}_{\tilde{\mathbf{L}}^F}\) and thus \(\tilde{b}\) are nilpotent.
\end{proof}

\begin{lem}\label{zh} Keep the notation in the proof of Lemma \ref{Good} and write
\(\mathbf C=C_{\tilde{\mathbf G}^{*}}(\tilde s)\).  If the \(p\)-element
\(\tilde t\) lies in \({\bf Z}(\mathbf C)^F\) (equivalently, in \({\bf Z}(\mathbf C^F)\)), then we have
\(
 {J}_{\tilde t}^{\tilde{\mathbf L}}={J}_{1}^{\tilde{\mathbf L}}.
\)
\end{lem}

\begin{proof}
Since \(\tilde s\) and \(\tilde t\) have coprime orders and
\(\tilde t\in {\bf Z}(\mathbf C)\), we have
\(C_{\tilde{\mathbf G}^{*}}(\tilde s\tilde t)
 =C_{\mathbf C}(\tilde t)=\mathbf C.\)
Let \(\tilde c\) be a unipotent \(p\)-block of \(\mathbf C^F\), and set
\(e=e_p(q)\), where \(e_p(q)\) is the order of \(q\) modulo \(p\) if \(p\) is
odd and the order of \(q\) modulo \(4\) if \(p=2\); this is the convention of
\cite[Section~2.4]{KM2}.  By \cite[Theorems~A and A.bis]{En}, \(\tilde c\) is
labelled by a \(\mathbf C^F\)-class of unipotent \(e\)-cuspidal pairs
\((\tilde{\mathbf M}^{*},\lambda)\) of quasi-central \(p\)-defect, and this class
is unique.

We now compare the two constructions in \cite[Proposition~4.1]{KM2}. Since \(C_{\tilde{\mathbf G}^{*}}(\tilde s\tilde t)
 =C_{\mathbf C}(\tilde t)=\mathbf C\), the constructions both for \(\tilde t\) and for \(1\) start with the
same group \(\mathbf C\), the same block \(\tilde c\), and hence the same
\(\mathbf C^F\)-class of pairs \((\tilde{\mathbf M}^{*},\lambda)\).  Since this is
already an \(e\)-cuspidal pair of \(\mathbf C\),
\cite[Proposition~3.6(a)]{KM2} and its uniqueness assertion show that both
relations \(\mathord{\to}_{\tilde t}\) and \(\mathord{\to}_{1}\) leave this class
fixed.  The subsequent natural bijection of \cite[Proposition~2.2]{KM2} is
independent of \(\tilde t\).  Thus \cite[Proposition~4.1(a),(b)]{KM2} assigns the
same block of \(\tilde{\mathbf L}^{F}\) to \(\tilde c\) in the two cases.  Since
this holds for every unipotent block \(\tilde c\) of \(\mathbf C^F\), we obtain
\({J}_{\tilde t}^{\tilde{\mathbf L}}={J}_{1}^{\tilde{\mathbf L}}\).
\end{proof}

\begin{lem}\label{Badquasi}
Assume that $p$ is bad for $\mathbf{G}$, that \(\mathbf{G}\) is of exceptional type \(G_2\), \(F_4\),\linebreak[0]
\(E_6\), \(E_7\) or \(E_8\), that \(b\) is quasi-isolated and that \(\IBr(b)\) is {\it a single \(\Aut(G)_b\)-orbit}. Then \(b\) has trivial defect group.
\end{lem}

\begin{proof}
This follows directly from {the second paragraph} of the proof of \cite[Proposition 3.11]{MNS}.
\end{proof}

\begin{prop}\label{All}
Let \(s\in G^*\) and \(\tilde s\in\tilde G^*\) be semisimple
\(p'\)-elements such that \(\iota^*(\tilde s)=s\). Let \(b\) and
\(\tilde b\) be blocks of \(G\) and \(\tilde G\), respectively, lying
in \(\mathcal E_p(G,s)\) and \(\mathcal E_p(\tilde G,\tilde s)\), and
assume that \(\tilde b\) covers \(b\). If
\(\IBr(b)\) is a single
\(\Aut(G)_b\)-orbit, then either \(b\) or \(\tilde b\)
is nilpotent with abelian defect groups.
\end{prop}

\begin{proof}
If \(p\) is good for \(\mathbf G\), the assertion follows from
Lemma~\ref{Good}. We may therefore assume that \(p\) is bad for
\(\mathbf G\).

Suppose first that \(b\) is quasi-isolated. If \(p=2\) and
\(\mathbf G\) has type \(B\), \(C\), or \(D\), then the first
paragraph of the proof of \cite[Lemma~5.2]{EKK} shows that \(b\) is
the principal block. Since $\IBr(b)$ is a single $\Aut(G)_b$-orbit, $\IBr(b)$ consists of only the trivial Brauer character. This forces $G$ to be $2$-nilpotent. A contradiction arises. In all remaining cases,
\(\mathbf G\) is of exceptional type and Lemma~\ref{Badquasi}
shows that \(b\) has defect zero.

We may now assume that \(b\) is not quasi-isolated. Let
\(\mathbf L\) be the minimal Levi subgroup of \(\mathbf G\) such that
\(
 C_{\mathbf G^*}(s)\leq \mathbf L^*
\).
If \(p\) is good for \(\mathbf L\), then
Lemma~\ref{Good}, applied to \(b\) and \(\tilde b\) with this
choice of \(\mathbf L\), shows that \(\tilde b\) is nilpotent
with abelian defect groups. Hence we may assume that \(p\) is bad
for \(\mathbf L\).

Let \(b_{\mathbf L^F}\) be the
Bonnaf\'e-Rouquier correspondent of \(b\).
The block \(b_{\mathbf L^F}\) is quasi-isolated and the proof of
\cite[Lemma~3.5]{MNS} shows that
\(\IBr(b_{\mathbf L^F})\) is a single
\(\Aut(\mathbf L^F)_{b_{\mathbf L^F}}\)-orbit.
By \cite[Proposition~12.14]{MT}, the components of
\([\mathbf L,\mathbf L]\) are simply connected. We take representatives
\(\mathbf L_1,\ldots,\mathbf L_t\) of the \(F\)-orbits on these
components and write
\[
 [\mathbf L,\mathbf L]^F=L_1\times\cdots\times L_t\,\,\mbox{and}\,\, L_i=\mathbf L_i^{F^{m_i}}.
\]
Set \(N=[\mathbf L,\mathbf L]^F\), and choose a block
\(d_N=d_1\otimes\cdots\otimes d_t\) of \(N\) covered by
\(b_{\mathbf L^F}\), where \(d_i\) is a block of \(L_i\).
Every \(d_i\) is
quasi-isolated.
Since \([\mathbf L,\mathbf L]\) is simply connected,
\cite[Remark~1.5.13, p.~62]{GM} shows that
\(N\) is characteristic in \(\mathbf L^F\). Hence
Lemma~\ref{Fong} applied to \(N\trianglelefteq\mathbf L^F\)
shows that \(\IBr(d_N)\) is a single
\(\Aut(N)_{d_N}\)-orbit.

Since \(p\) is bad for \(\mathbf L\), choose \(i_0\)
such that \(p\) is bad for \(\mathbf L_{i_0}\). If
\(\mathbf L_{i_0}\) is of classical type, then it has type \(B\),
\(C\), or \(D\), and \(p=2\). Since \(p\ne\ell\), the defining
characteristic of \(\mathbf L_{i_0}\) is odd, so \(L_{i_0}\) is
quasisimple. As \(N=L_1\times\cdots\times L_t\), we have
\(L_{i_0}\trianglelefteq N\); hence \(L_{i_0}\) is a component
of \(N\). Choose a block \(e\) of \(\mathbf E(N)\) covered by
\(d_N\) and covering \(d_{i_0}\). Lemma~\ref{Normal-one-orbit2},
applied to the block \(d_N\) of \(N\), now shows that
\(\IBr(d_{i_0})\) is a single
\(\Aut(L_{i_0})_{d_{i_0}}\)-orbit. The second paragraph excludes
this possibility. Since every prime is good for type \(A\),
\(\mathbf L_{i_0}\) must therefore be exceptional. Thus
\([\mathbf L,\mathbf L]\) has an exceptional component.

The possible proper Levi subgroups with an exceptional component are
obtained by deleting vertices from the exceptional Dynkin diagrams:
\[
\begin{array}{c|c}
 \text{type of }\mathbf G&
 \text{possible type of }[\mathbf L,\mathbf L]\\ \hline
 G_2,\ F_4,\ E_6&\text{none}\\
 E_7&E_6\\
 E_8&E_7,\ E_6,\ E_6+A_1.
\end{array}
\]
Consequently, either \([\mathbf L,\mathbf L]\) is simple, or
\(\mathbf G\) has type \(E_8\) and
\([\mathbf L,\mathbf L]\) has type \(E_6+A_1\).

Suppose first that \([\mathbf L,\mathbf L]\) is simple and let \(d\)
be its block covered by \(b_{\mathbf L^F}\).
In this case \([\mathbf L,\mathbf L]\) has type
\(E_6\) or \(E_7\), so \([\mathbf L,\mathbf L]^F\) is quasisimple
and is a component of \(\mathbf L^F\). Thus the quasi-isolated
block \(d\) satisfies the orbit condition by
Lemma~\ref{Normal-one-orbit2}. Since \(p\) is bad for
\([\mathbf L,\mathbf L]\), Lemma~\ref{Badquasi} therefore implies
that \(d\) has defect zero. Since
\(\mathbf L^F/[\mathbf L,\mathbf L]^F\) is abelian,
\cite[Proposition~6.5]{KP} and \cite[Lemma~3.4]{CEKL} show that
\(b_{\mathbf L^F}\) is nilpotent with abelian defect groups. The
Bonnaf\'e--Rouquier equivalence gives the same conclusion for \(b\).

It remains to consider the case \(\mathbf G\) has type \(E_8\) and
\(
 [\mathbf L,\mathbf L]=\mathbf L_1\times\mathbf L_2,
\)
where \(\mathbf L_1\) and \(\mathbf L_2\) have types \(E_6\) and
\(A_1\), respectively. Since \(p\) is bad for \(\mathbf L\), we have
\(p\in\{2,3\}\). We write $d=d_1 \otimes d_2$, where \(d_i\) is a quasi-isolated block of \(\mathbf{L}_i^F\) covered by \(b_{\mathbf{L}^F}\).
By Lemma~\ref{Badquasi}, \(d_1\) has defect zero.

Since $\mathbf G$ has connected center, $\mathbf Z(\mathbf L)$ is
connected (see
\cite[Lemma~13.14(ii)]{DM}). Then by
\cite[Theorem~4.5.9]{C}, we have
\(
 C_{\mathbf L^*}(s)=C_{\mathbf L^*}^{\circ}(s)
\) and hence $s$ is
isolated in $\mathbf L^*$.
The inclusion
\(
 [\mathbf L,\mathbf L]\hookrightarrow\mathbf L
\)
is a regular embedding. Let $t$ be the image of $s$ under the dual
isotypic morphism
\(
 \mathbf L^*\rightarrow[\mathbf L,\mathbf L]^*.
\)
By \cite[Proposition~2.3(b)]{Bon1}, $t$ is isolated. Moreover, it follows from \cite[Proposition 15.6]{CE} that
\(
 \Irr(d)\cap
 \mathcal E([\mathbf L,\mathbf L]^F,t)\neq\varnothing.
\)
Write $t=(t_1,r)$ under the decomposition
\(
 [\mathbf L,\mathbf L]^*
   =\mathbf L_1^*\times\mathbf L_2^*.
\)
Since $t$ is isolated, both $t_1$ and $r$ are isolated in their
respective factors. Moreover, from $d=d_1\otimes d_2$ and the product
decomposition of Lusztig series, we obtain
\(
 \Irr(d_2)\cap
 \mathcal E(\mathbf L_2^F,r)\neq\varnothing.
\)
Now $\mathbf L_2$ is simply connected of type $A_1$, so
$\mathbf L_2^*$ is of adjoint type $A_1$. The identity is the only
isolated semisimple element of $\mathbf L_2^*$. Hence $r=1$, and
therefore $d_2$ is unipotent.

Assume first that \(p=3\). By \cite[Lemma~3.2]{MNS} and
\cite[Theorem~14.6]{CE}, the unipotent characters in \(d_2\) form a
single automorphism orbit and hence have equal degrees.
Lemma~\ref{unipotentdegree} shows that \(d_2\) has central defect.
Since \(p=3\) does not divide
\(\lvert\mathbf Z(\mathbf L_2^F)\rvert\), the block \(d_2\) has
defect zero. Thus \(d=d_1\otimes d_2\) has defect zero and the same
argument as in the simple case shows that \(b\) is nilpotent with
abelian defect groups.

Finally, assume that \(p=2\). The unipotent block \(d_2\) is the
 principal \(2\)-block of \(\mathbf L_2^F\). Since $\IBr(d_2)$ is a single $\Aut(\mathbf L_2^F)_{d_2}$-orbit, we have
\(
 \IBr(d_2)=\{1_{\mathbf L_2^F}\}.
\)
So \(\mathbf L_2^F\) is \(2\)-nilpotent, which is
impossible for a finite group of Lie type \(A_1\) in odd defining
characteristic. Hence this case cannot occur, completing the proof.
\end{proof}

\section{Proof of Theorem \ref{Main}}

In this section, we give a proof of Theorem~\ref{Main}. By Proposition~\ref{Reduction},
Theorem~\ref{Main} is equivalent to the following proposition.

\begin{prop}\label{Pro:Quai-simple}
 Let  \(H\) be a quasisimple group, \(\tilde{H}\) a finite group with \(H \trianglelefteq \tilde{H}\),
 \(c\) a block of \(H\), and \(\tilde{c}\) a block of \(\tilde{H}\) covering \(c\) with an abelian defect group \(\tilde{P}\).
Assume that  \(\tilde{H}=\tilde{P}H\). If \(\IBr(c)\) is a single \(\Aut(H)_{c}\)-orbit,
 {then} \(\tilde{c}\) is inertial.
 \end{prop}

Observe that we have \(\mathbf{O}_p(\tilde{H}) = \mathbf{O}_p(\mathbf{Z}(\tilde{H}))\) and \(\mathbf{O}_p(\tilde{H}) \cap H = \mathbf{O}_p(\mathbf{Z}(H))\). Set \(\tilde{L}= \tilde{H}/\mathbf{O}_p(\tilde{H})\) and \(L= H/\mathbf{O}_p(\mathbf{Z}(H))\). Then $L$ may be identified with a normal subgroup of \(\tilde L\). Let \(\tilde{f}\) and \(f\) be the blocks of \(\tilde{L}\) and \(L\) dominated by \(\tilde{c}\) and \(c\) respectively. Then \(\IBr(c) = \IBr(f)\) and
\(\IBr(f)\) remains a single \(\Aut(L)_{f}\)-orbit. By \cite[Corollary 1.14]{P1}, \(\tilde{c}\) is inertial if and only if \(\tilde{f}\) is inertial. Consequently Proposition \ref{Pro:Quai-simple} is equivalent to the following statement.

\begin{prop}\label{Pro:Quai-simple-1}
 Let  \(H\) be a quasisimple group, \(\tilde{H}\) a finite group with \(H \trianglelefteq \tilde{H}\),
 \(c\) a block of \(H\), and \(\tilde{c}\) a block of \(\tilde{H}\) covering \(c\) and having an abelian defect group \(\tilde{P}\).
Assume that  \(\tilde{H}=\tilde{P}H\) and that $\mathbf{O}_p(\tilde{H})=1$. If \(\IBr(c)\) is a single \(\Aut(H)_{c}\)-orbit,
 then \(\tilde{c}\) is inertial.
 \end{prop}

\begin{lem}\label{Al}
Proposition~\ref{Pro:Quai-simple-1} holds when
one of the following two cases occurs.

\smallskip\noindent{\bf 4.3.1.}
\(H/\mathbf Z(H)\) is a sporadic simple group, an alternating group
\(\mathfrak A_n\) with \(n\geq5\), a simple group of Lie type in
defining characteristic \(p\), or the Tits group \({}^2F_4(2)'\).

\smallskip\noindent{\bf 4.3.2.}
\(H\) is an exceptional covering group of Lie type and \(c\) is
faithful.
\end{lem}

\begin{proof}
If \(c\) is nilpotent, then so is \(\tilde c\) and there is nothing to
prove. So we may assume that \(c\) is not nilpotent. Set
\(S=H/\mathbf Z(H)\).

Suppose first that Statement~{\bf 4.3.2} holds. By
\cite[Proposition~2.11]{MNS}, we have
\(
 H\cong 2.G_2(4)\), \(p=3\), \(|\IBr(c)|=2
\)
and \(c\) has a defect group of prime order. Moreover,
\(\operatorname{Out}(H)\) is cyclic of order \(2\); see the Atlas
\cite{CCNPW}.

Suppose instead that Statement~{\bf 4.3.1} holds. By
\cite[Proposition~3.1]{MNS}, \(S\) is not of Lie type in defining
characteristic \(p\). By \cite[Proposition~2.6 and Theorem~2.8]{MNS},
the block \(c\) has a defect group of prime order. Since a \(2\)-block
with cyclic defect groups is nilpotent, \(p\geq3\). Furthermore,
\cite[Corollary~11.1.4]{L} gives
\(|\IBr(c)|\geq2\). Since \(\IBr(c)\) is a single \(\Aut(H)_{c}\)-orbit, we have
\(|\operatorname{Out}(H)|\geq2\). The classification in
\cite[Proposition~2.6 and Theorem~2.8]{MNS} now yields
\[
 (H,p)\in\bigl\{(2.J_2,3),(2.HS,3),
 (2.\mathfrak A_n,3)\ (n\neq 6),(2.\mathfrak A_6,5),
 (6.\mathfrak A_6,5)\bigr\},
\]
and in every case \(|\IBr(c)|=2\). By
\cite[Theorem~5.2.1 and Tables~5.3g and 5.3m]{GLS},
\(\operatorname{Out}(H)\) is cyclic of order \(2\), except when
\(S\cong\mathfrak A_6\), in which case it is a Klein four group.

In either case, \(p\in\{3,5\}\),
\(|\IBr(c)|=2\), the defect groups of \(c\) have prime
order and \(p\nmid|\operatorname{Out}(H)|\). Since
\(\tilde H/H\) is a \(p\)-group, the $\tilde H$-conjugation action on \(H\) produces
trivial image in \(\operatorname{Out}(H)\). Hence
\(\tilde H=H C_{\tilde H}(H)\) and
\(C_{\tilde H}(H)/\mathbf Z(H)\ \text{is a \(p\)-group}.
\)
The condition \(\mathbf O_p(\tilde H)=1\) implies that
\(\mathbf Z(H)\) is a \(p'\)-group. Set \(C=C_{\tilde H}(H)\). Then
\(C\trianglelefteq\tilde H\), \(\mathbf Z(H)\leq\mathbf Z(C)\) and
\(C/\mathbf Z(H)\) is a \(p\)-group. By the Schur-Zassenhaus
theorem, \(C\) has a unique complement \(Q\) to \(\mathbf Z(H)\). Thus \(Q\) is the unique Sylow
\(p\)-subgroup of \(C\), so it is characteristic in \(C\) and hence
normal in \(\tilde H\). Therefore \(Q=1\). Consequently, we have
\(C_{\tilde H}(H)=\mathbf Z(H)\) and \(\tilde H=H\).

It remains to prove that \(c\) is inertial. Let \(D\) be a defect
group of \(c\). If \(p=3\), then \(D\) is cyclic of order 3 and it is known that every block with
defect group of order 3 is inertial.

Suppose now that \(p=5\). Then \(D\) is cyclic of order 5 and
\(|\IBr(c)|=2\). Hence the inertial index of \(c\) is
\(e=2\) and its exceptional multiplicity is
\(
 m=\frac{|D|-1}{e}=2.
\)
In particular, the exceptional vertex
of the Brauer tree \(\Gamma\) of \(c\) is uniquely distinguished.
By the orbit hypothesis there is an automorphism
\(\alpha\in\Aut(H)_c\) interchanging the two members of
\(\IBr(c)\). The induced automorphism of \(\Gamma\)
therefore interchanges its two edges and fixes its exceptional
vertex. A tree with two edges is a path of length two, and its
edge-interchanging automorphism fixes precisely the middle vertex.
Thus the exceptional vertex is the middle vertex of \(\Gamma\) and
\(\Gamma\) is a star with its exceptional vertex at the center.

Let \(c_D\) be the Brauer correspondent of \(c\) in \(N_H(D)\). Its
Brauer tree is the two-edge star with exceptional multiplicity \(2\)
and the exceptional vertex at the center; see
\cite[Theorem~11.2.2]{L}. Thus \(c\) and \(c_D\) have the same planar
Brauer tree and exceptional multiplicity. It follows from
\cite[Theorem~11.1.16]{L} that they are Morita equivalent, and
\cite[Remark~11.14.4]{L} shows that this equivalence is basic.
Therefore \(c\) is inertial. We have proved in both cases that
\(c\) is inertial.
\end{proof}

\begin{lem}\label{2B2}
Proposition \ref{Pro:Quai-simple-1} holds when \(H\) is a simple group \({}^2B_2(2^{2n+1})\), \({}^2G_2(3^{2n+1})\), or \({}^2F_4(2^{2n+1})\) for some \(n \geq 1\).
\end{lem}

\begin{proof}
Suppose \(p = 2\). By \cite[\S 4]{W}, the outer automorphism group
\(\operatorname{Out}(H)\) is cyclic of odd order \(2n+1\). Hence the
image of the \(2\)-group \(\tilde H/H\) in \(\operatorname{Out}(H)\)
is trivial. The centralizer argument in the proof of Lemma~\ref{Al},
together with \(\mathbf O_2(\tilde H)=1\), therefore gives
\(\tilde H=H\). If
\(H \cong {}^2B_2(2^{2n+1})\) or
\({}^2F_4(2^{2n+1})\), then Lemma~\ref{Al} shows that \(c\) is
inertial. Suppose now that \(H \cong {}^2G_2(3^{2n+1})\) and that
\(c\) is non-nilpotent. Since the principal block of \(H\) is the
unique block with an elementary abelian defect group of order \(8\)
\cite[Theorem~1.3]{KK} and \(c\) is non-principal, its defect group is
a Klein four group, \(|\IBr(c)|=3\), and all its
irreducible Brauer characters have height zero. By
\cite[Theorem~12.1.2]{L} and the argument following
\cite[Corollary~12.1.5]{L}, \(c\) is basically Morita equivalent to
\(\mathcal O\mathfrak A_4\), and hence is inertial.

Now suppose \(p > 2\). We will show that \(c\) is nilpotent. Then $\tilde c$ is nilpotent.

First let \(H = {}^2G_2(3^{2n+1})\). By Lemma \ref{Al}, we may assume \(p > 3\). Sylow \(p\)-subgroups of \(H\) are cyclic. By the Brauer trees described in \cite[\S 4.1]{HG}, the number of irreducible Brauer characters in $c$ is either even or exactly one. Since \(\IBr(c)\) is a single \(\Aut(H)_c\)-orbit and \(|\mathrm{Out}(H)| = 2n + 1\), we have \(|\IBr(c)| = 1\). So \(c\) is nilpotent.

Next let \(H = {}^2B_2(2^{2n+1})\). Sylow \(p\)-subgroups of \(H\) are cyclic. Since \(\IBr(c)\) is a single \(\Aut(H)_c\)-orbit, \(c\) is not principal. By \cite[\S 2]{Bu}, all irreducible ordinary characters in \(c\) have equal degree. By \cite[Theorem 4.1]{MN},  \(c\) is nilpotent.

Finally{,} let \(H = {}^2F_4(2^{2n+1})\). We may identify $H$ with its dual and let \(s \in H\) be a {semisimple} \(p'\)-element such that \(\Irr(c) \cap \mathcal{E}(H, s) \neq \emptyset\).

Assume that $p = 3$ and $c$ is unipotent.  By the proof of \cite[Proposition 3.11]{MNS}, $c$ is of defect zero.

Assume that $p\neq 3$ or that $c$ is not unipotent.
We claim that \(|\IBr(c)|=1\).
Suppose that $c$ is not unipotent. By the proof of \cite[Lemma 3.5]{MNS}, \(\Irr(c)\cap\mathcal{E}(H, s)\) is \(\Aut(H)_c\)-stable. When $p>3$, by \cite[Theorem 14.4]{CE}, \(\Irr(c)\cap\mathcal{E}(H, s)\) is a basic set of the block $c$. When $p=3$, by the proof of \cite[Proposition 3.11]{MNS}, \(1\) is the unique quasi-isolated semisimple \(3'\)-element, hence $s$ is not quasi-isolated. Then it follows from \cite[Corollary 3.8]{MNS} that
\(\Irr(c)\cap\mathcal{E}(H, s)\) is a basic set of the block $c$.
Since \(\IBr(c)\) is a single \(\Aut(H)_c\)-orbit, by \cite[Lemma 3.2]{MNS}, \(\Aut(H)_c\) has only one orbit on \(\Irr(c)\cap\mathcal{E}(H, s)\). In particular, all characters in \(\Irr(c)\cap\mathcal{E}(H, s)\) have the same degree. But
\cite[Table A.2]{Hi} {shows that} \(\mathcal{E}(H, s)\) {has} at most two characters of equal degree. Therefore{,} \(|\IBr(c)|=|\Irr(c)\cap\mathcal{E}(H, s)|\leq 2\).
Since \(|\mathrm{Out}(H)|=2n+1\) is odd, we have \(|\IBr(c)|=1\).
Suppose $p>3$ and $c$ is unipotent. Since $p$ is good for $H$, as above, we prove that \(\Aut(H)_c\) has only one orbit on \(\Irr(c)\cap\mathcal{E}(H, s)\).
By \cite[Theorem 4.5.11]{GM}, we have \(|\IBr(c)|=1\). {This proves the claim.}

By \cite[Corollary 11.1.4]{L}, $c$ is nilpotent if $c$ is a cyclic block. In the remainder, we assume that $c$ is not a cyclic block. We adopt the notation for types of \(s\) from Table~IV in \cite{Sh}.
Suppose that $p$ does not divide \((2^{2n+1} - 1)\) or that \(s\) is {of neither type \(t_0=1\) nor type \(t_1\)}.
By \cite[2C(a)]{A1}, we have $\mathcal{E}(H, s) \subseteq \Irr(c)$.
{As shown in the preceding paragraph}, we have $|\mathcal{E}(H,s)|=1$. By \cite[Theorem 2.6.4]{GM}, $C_H(s)$ has only one unipotent character, hence \(C_H(s)\) is a torus. Now, by \cite[Lemma 4.2]{EKK}, \(c\) is nilpotent with abelian defect groups.

Suppose  \(p \mid (2^{2n+1} - 1)\) and that \(s\) is of type \(t_0=1\) or \(t_1\). We have $|\Irr(c) \cap \mathcal{E}(H, s)|=1$ as above.
By \cite[Bemerkung 1(c)]{M}, $c$ is not unipotent.
Then \(s\) is of type \(t_1\) and it follows from \cite[Page 10]{Sh} that we have \(C_H(s) \cong {}^2B_2(2^{2n+1}) \times \mathbb{Z}_{2^{2n+1}-1}\).
Let \(d\) be the block of \(C_H(s)\) corresponding to \(c\) through the Bonnaf\'e-Rouquier correspondence.
Then \(|\IBr(d)|=1\) and hence the block of \({}^2B_2(2^{2n+1})\) dominated by $d$ also has exactly one irreducible Brauer character. The \({}^2B_2(2^{2n+1})\) case implies that the latter block is nilpotent with abelian defect groups and thus that \(d\) is nilpotent with abelian defect groups. So \(c\) is nilpotent with abelian defect groups.
\end{proof}

\begin{proof}[\rm\bf Proof of Proposition \ref{Pro:Quai-simple-1}]

We keep the notation and the assumptions in Proposition \ref{Pro:Quai-simple-1}.
By {Lemmas~\ref{Al}--\ref{2B2}}, we may assume that there is a simple simply connected algebraic group over $\overline{\mathbb{F}}_\ell$ (\(\ell \neq p\)) with a Frobenius endomorphism \(F\) such that $H=G/Z$ for some subgroup $Z\leq \mathbf{Z}(G)$. Since \(\mathbf{O}_p(\tilde{H}) \cap H = 1\), we have \(\mathbf{O}_p(G) \leq Z\). Let \(b\) be the block of \(G\) dominating \(c\). By \cite[Corollary B.8]{N2}, we have \(\Aut(H) \cong \Aut(G)_Z\) and therefore \(\IBr(b)\) is a single \(\Aut(G)_b\)-orbit.

Let \(\iota:\bG\hookrightarrow \tilde{\bf G}\) be a regular embedding.
Identifying \(\bG\) with its image \(\iota(\bG)\), we have \(\tilde{\bf G}=\bG\mathbf{Z}(\tilde{\bf G})\).
The Frobenius endomorphism \(F\) on \(\bG\)
may be extended to a Frobenius endomorphism on \(\tilde{\bf G}\), which we again denote by \(F\).
Set \(\tilde{G}=\tilde{\bf G}^F\) and
let \(\tilde{b}\) be a block of \(\tilde{G}\) covering \(b\).
If \(b\) is nilpotent, then $c$ and hence $\tilde c$ are nilpotent. Otherwise, by Proposition \ref{All}, \(\tilde{b}\) is nilpotent.

Denote by \(\operatorname{Ker}(c)\) the kernel of the block \(c\). {Since $\tilde P$ is a defect group of $\tilde c$, it stabilizes $c$. Thus $c$ is $\tilde H$-invariant and, as $\tilde H/H$ is a $p$-group, $c$ is also the unique block of $\tilde H$ covering $c$; we therefore identify it with $\tilde c$.} {Clifford theory yields} $\operatorname{Ker}(c)\leq\operatorname{Ker}(\tilde c)$. Conversely, $\operatorname{Ker}(\tilde c)$ is a $p'$-group and hence is contained in $H$. {Another application of Clifford theory gives} $\operatorname{Ker}(\tilde c)\leq\operatorname{Ker}(c)$. {Therefore},
we have $\operatorname{Ker}(c)=\operatorname{Ker}(\tilde c)$.

Set $\bar H=H/\operatorname{Ker}(c)$ and $\bar{\tilde H}=\tilde H/\operatorname{Ker}(\tilde c)$ and denote by
$\bar c$ and $\bar {\tilde c}$ the images of $c$ and $\tilde c$ in $\bar H$ and $\bar{\tilde H}$. The natural homomorphisms $H\rightarrow \bar H$ and $\tilde H\rightarrow \bar{\tilde H}$ induce algebra isomorphisms
\(\mathcal{O}Hc \cong \mathcal{O}\bar{H}\,\bar{c}\) and \(\mathcal{O}\tilde H\tilde c \cong \mathcal{O}\bar{\tilde H}\,\bar{\tilde c}\). Thus, in order to prove that \(\tilde{c}\) is inertial, we may assume that \(\operatorname{Ker}(\tilde{c}) = \operatorname{Ker}(c) = 1\). Then, since the block \(b\) dominates the block \(c\), we have \(\operatorname{Ker}(b) = \mathbf{O}_{p'}(Z)\).

By \cite[Theorem 2.5.1]{GLS}, \(\Aut(G)\) is induced by \(\tilde{G} \rtimes \mathcal{D}\) and \(C_{\tilde{G}}(G) = \mathbf{Z}(\tilde{G})\), where \(\mathcal{D}\) denotes the group generated by suitable graph and field automorphisms of \(G\). In particular, \(\Aut(G)\) is isomorphic to \((\tilde{G}/\mathbf{Z}(\tilde{G})) \rtimes \mathcal{D}\). Since \(H = G/Z\), by \cite[Corollary B.8]{N2}, \(\Aut(H)\) is induced by \(\tilde{G} \rtimes \mathcal{D}_Z\) and is isomorphic to \((\tilde{G}/\mathbf{Z}(\tilde{G})) \rtimes \mathcal{D}_Z\), where \(\mathcal{D}_Z\) denotes the stabilizer of \(Z\) in \(\mathcal{D}\).

Set \(\check{G} = \tilde{G}/(Z\mathbf{O}_{p}(\mathbf{Z}(\tilde{G})))\) and \(\check{A}= (\tilde{G }\rtimes \mathcal{D}_Z)/(Z\mathbf{O}_{p}(\mathbf{Z}(\tilde{G})))\). Since we have \(\mathbf{O}_p(\mathbf{Z}(\tilde{G})) \cap G \subseteq \mathbf{O}_p(G)\subseteq Z \),
the inclusion $G\subseteq G\mathbf{O}_p(\mathbf{Z}(\tilde{G}))$ induces an isomorphism
$$
H = G/Z \cong (G\mathbf{O}_p(\mathbf{Z}(\tilde{G})))/(Z\mathbf{O}_p(\mathbf{Z}(\tilde{G}))),
$$
through which we may identify \(H\) with a normal subgroup of \(\check{A}\).
By Lemma \ref{embedding} below, we may identify $\tilde H$ with a subgroup of $\check A$ so that $H$ is normal in $\check A$.

Since \(\tilde{G}/G\) is abelian, so is \(\check G/H\). Let \(\check K\) be the normal subgroup of \(\check G\) containing \(H\) such that \(\check K/H\) is a \(p'\)-group and \(\check G/\check K\) is a \(p\)-group. Since \(\operatorname{Ker}(b) = \mathbf{O}_{p'}(Z)\), we have \(Z\mathbf{O}_{p}(\mathbf{Z}(\tilde{G})) = \operatorname{Ker}(b)\mathbf{O}_{p}(\mathbf{Z}(\tilde{G}))\). Since the block \(\tilde{b}\) covers \(b\), we have \(\operatorname{Ker}(\tilde{b}) \geq \operatorname{Ker}(b)\). Hence \(\tilde{b}\) dominates the unique block \(\check b\) of \(\check G\). Since \(\tilde{b}\) covers \(b\) and \(b\) dominates \(c\), \(\check b\) covers \(c\). Let \(\check d\) be a block of \(\check K\) that covers \(c\) and is covered by \(\check b\). Since \(\tilde{b}\) is nilpotent, so are \(\check b\) and \(\check d\). Let \(\operatorname{Bl}(\check K|c)\) be the set of blocks of \(\check K\) covering \(c\). Since \(\check K/H\) is abelian, by \cite[Lemma 2.2 (a)]{KM}, the cardinality of \(\operatorname{Bl}(\check K|c)\) is a \(p'\)-number and every block in \(\operatorname{Bl}(\check K|c)\) is nilpotent.

Clearly $\tilde P$ normalizes $\check G$ and $\check K$, thus \(\check K\) is normal in \(\tilde{P}\check K\). Since \(\tilde{P}\) stabilizes \(c\), it stabilizes the set \(\operatorname{Bl}(\check K|c)\) and hence a block \(\check f\) in \(\operatorname{Bl}(\check K|c)\). In particular, this block \(\check f\) is a block of \(\tilde{P}\check K\) with defect group \(\tilde{P}\) and it is nilpotent. Set \(L = HN_{\tilde P\check K}(\tilde{P})\) and let \(e\) be the Brauer correspondent of \(\check f\) in \(L\). Since \(\tilde{P}\) is abelian, the Brauer categories of \(\check f\) and \(e\) are the same. Hence the block \(e\) is nilpotent. Moreover, \(e\) covers the block \(\tilde{c}\). By \cite[Theorem 3.13]{P2}, the block \(\tilde{c}\)  is inertial.
\end{proof}

\begin{lem}\label{embedding}
Keep the notation and the assumptions in Proposition~4.2 and its proof. Then there exists an injective homomorphism $\tilde H\rightarrow \check A$
whose restriction to \(H\) is the identity.
\end{lem}

\begin{proof}
Let \(\tau\colon\tilde H\to\Aut(H)\) be induced by conjugation. We first claim that
$C_{\tilde H}(H)={\mathbf Z(H)}$.
Indeed, \(C_{\tilde H}(H)/{\mathbf Z(H)}\) is a \(p\)-group whereas {\(\mathbf Z(H)\)}
is a central \(p'\)-subgroup. By the Schur--Zassenhaus theorem,
\(C_{\tilde H}(H)\) has a normal Sylow \(p\)-subgroup, which is
contained in \({\bf O}_p(\tilde H)=1\). The claim follows. Consequently,
$
    \tau|_{\tilde P}$ is injective
  and
    $\tau(\tilde P)\cap\Inn(H)=\tau(P)$.

Let
$
    \rho\colon\check A\rightarrow\Aut(H)
$
be the natural homomorphism induced by {the conjugation action of $\check A$}. The kernel of $\rho$ is a \(p'\)-group. Choose a Sylow $p$-subgroup $\check P$ of $\rho^{-1}(\tau(\tilde P))$, which clearly contains
\(P=\tilde P\cap H\). Then \(\rho|_{\check P}\) is an isomorphism onto
\(\tau(\tilde P)\). Hence
$
    \lambda
      =(\rho|_{\check P})^{-1}\circ\tau|_{\tilde P}
      \colon\tilde P\xrightarrow{\sim}\check P
$
is identical on $P$. Since $\tau(\tilde P)\cap\Inn(H)=\tau(P)$, we have
\(\check P\cap H=P\).

Define
$
    \Phi:\tilde H\rightarrow\check A, {hx\mapsto h\lambda(x)}
$
for \(h\in H\) and \(x\in\tilde P\). Since \(\lambda\) fixes
\(P=H\cap\tilde P\) and \(\lambda(x)\) induces the
same automorphism on \(H\) as \(x\) does, the map \(\Phi\) is a well-defined
homomorphism. The equality \(\check P\cap H=P\) implies
\(\ker(\Phi)=1\). Thus \(\Phi\) is the desired homomorphism.
\end{proof}

Finally, we close this section with the proof of Proposition \ref{Quasi-one}.

\begin{proof}[\rm\bf Proof of Proposition \ref{Quasi-one}] The kernel $K=\operatorname{Ker}(b)$ is a central $p'$-subgroup of $G$.
The block $b$ dominates a faithful block $\bar{b}$ of $G/K$ and the defect
groups of $b$ and $\bar{b}$ are naturally isomorphic. Moreover,
$\Aut(G)_b$ stabilizes $K$ and induces a transitive action on
$\IBr(\bar{b})$. Hence we may replace $G$ by $G/K$ and assume
that $b$ is faithful.

Suppose that \(G/\mathbf{Z}(G)\) is one of the following: a sporadic simple group, an alternating group \(\mathfrak{A}_{n}\) (\(n \geq 5\)), a simple group of Lie type in the defining characteristic \(p\), or the Tits group \({}^2F_4(2)^{\prime}\). Then this proposition follows from \cite[Proposition 2.6, Theorem 2.8 and Proposition 3.1]{MNS}.
Suppose that \(G\) is an exceptional covering group of Lie type. Then this proposition follows from
\cite[Proposition 2.11]{MNS}.

Suppose that \(G\) is a simple group of the form \({}^2B_2(2^{2n+1})\), \({}^2G_2(3^{2n+1})\), or \({}^2F_4(2^{2n+1})\) for some integer \(n \geq 1\), and that \(p\) is not the defining characteristic of \(G\). If the Sylow \(p\)-subgroup of $G$ is abelian, then the defect group of \(b\) is abelian.
Thus we may assume that \(p = 3\) and \(G \cong {}^2F_4(2^{2n+1})\) with \(n \geq 1\).
By the proof of the case ${}^2F_4(2^{2n+1})$ in Lemma~\ref{2B2}, the defect group of \(b\) is abelian in this case.

By the classification of finite simple groups, for the proof of the remaining case, we may write
$G=G_{0}/Z$ and $G_{0}=\mathbf{G}^{F}$,
where $\mathbf{G}$ is a simple simply connected algebraic group, $F$ is {a} Frobenius endomorphism on $\mathbf{G}${,} and
$Z\leq {\mathbf Z(G_{0})}$. Let $b_{0}$ be the block of $G_{0}$ dominating $b$. It is known that the natural homomorphism
$G_{0}\rightarrow G$ induces a bijection between $\IBr(b_{0})$ and $\IBr(b)$.
By
\cite[Corollary~B.8]{N2}, the correspondence between $b$ and $b_{0}$ is
compatible with automorphisms; hence $\IBr(b_{0})$ is a single
$\Aut(G_{0})_{b_{0}}$-orbit.

Choose a regular embedding
\[
  \mathbf{G}\hookrightarrow\tilde{\mathbf{G}}
\]
with compatible Frobenius endomorphism $F$ and let $\tilde{b}_{0}$ be a block
of $\tilde{G}_{0}=\tilde{\mathbf{G}}^{\,F}$ covering $b_{0}$ in the
corresponding Lusztig series. By Proposition~\ref{All}, either $b_{0}$ or $\tilde{b}_{0}$ is
nilpotent with abelian defect groups. {In either case}, $b_{0}$ has {an} abelian defect group.
Since $b_0$ dominates $b$ through the natural homomorphism $G_0\rightarrow G$, the defect groups of $b$ are
images of defect groups of $b_{0}$ under the homomorphism
$G_{0}\to G$. {Hence} $b$ has {an} abelian defect group. {This completes the proof.}
\end{proof}

\section{Proof of Theorem \ref{Main1} and \ref{Main2}}
In this section, we give proofs of  Theorem \ref{Main1} and \ref{Main2}.

Let $G$ be a finite group and $b$ a block of $G$.
Let $Z\trianglelefteq G$ be a normal subgroup. A block $\bar{b}$ of the quotient group $G/Z$ is said to be contained in $b$ if
\(
\IBr(\bar{b})\subseteq\IBr(b),
\)
where the irreducible Brauer characters in $\IBr(\bar{b})$ are inflated to characters of $G$. For $\theta\in \Irr(Z)$, we denote by $\Irr(G\mid\theta)$ the set of irreducible characters of $G$ lying above~$\theta$. Similarly, for any $\psi\in \IBr(Z)$, we write {$\IBr(G\mid\psi)$} for the set of irreducible Brauer characters of $G$ lying above~$\psi$.
Given any $\chi\in\Irr(G)\cup \IBr(G)$, we denote by $\operatorname{bl}(\chi)$ the unique block of $G$ containing $\chi$.
Let $\operatorname{dz}(G)$ be the set of {defect-zero} characters in $\Irr(G)$.
Let $Q$ be a $p$-subgroup of $G$ and let $N_G(Q)$ be its normalizer in~$G$. We denote by $\operatorname{dz}(N_G(Q)/Q,\, b)$ the set of characters $\bar \chi$ in $\operatorname{dz}(N_G(Q)/Q)$ such that viewing $\bar \chi$ as a character of $N_G(Q)$, we have $\operatorname{bl}(\bar\chi)^G = b$.
Set $l(b)=|\IBr(b)|$ and denote by $\ell(b)$
the number of orbits of the $G$-action on the set of $p$-weights of $b$. The blockwise Alperin weight conjecture predicts $l(b)=\ell(b)$.
Fixing a defect group $P$ of $b$ and letting $\Gamma$ be a set of representatives for the $G$-conjugacy classes of $p$-subgroups of $P$,
we have the equality
\(\ell(b)=\sum_{Q\in\Gamma}\bigl|\operatorname{dz}(N_G(Q)/Q,\, b)\bigr|\).

\begin{lem}\label{Vertex}
Let $G$ be a finite group and let $b$ be a block of $G$ with defect group $P$.
Let $N$ be a normal subgroup of $G$ such that $b$ covers a $G$-invariant block $c$ of $N$.
Set $D = P \cap N$. Assume that $D$ is abelian and that $Q$ is a subgroup of $P$ such that ${\rm dz}(N_G(Q)/Q, b)$ is nonempty.
Then $D\leq Q$.
\end{lem}

\begin{proof}
Set \(H = QN\).
Take $\bar\chi\in {\rm dz}(N_G(Q)/Q, b)$ and an irreducible constituent $\bar \psi$ of the restriction of $\bar\chi$ to $N_H(Q)/Q$.
Regard $\bar\chi$ as a character of $N_G(Q)$ and $\bar\psi$ as a character of $N_H(Q)$.
Set \(b_Q = \operatorname{bl}(\bar\chi)\) and \(c_Q = \operatorname{bl}(\bar\psi)\). Then $b_Q$ covers $c_Q$.

There is some $x\in G$ such that $N_P(Q^x)$ is a defect group of the block $(b_{Q})^x$ of $N_G(Q^x)$.
If $D\leq Q^x$, then $D^{x^{-1}}\leq Q$. Since $c$ is $G$-invariant, $D^{x^{-1}}$ is a defect group of $c$. Moreover, $D^{x^{-1}}\leq Q\leq P$
and hence $D^{x^{-1}}\leq P\cap N=D$.
Equality of orders gives $D^{x^{-1}}=D$, so $D\leq Q$.
Consequently, after conjugating all the data if necessary, we may assume that
$R=N_P(Q)$
is a defect group of $b_Q$.
In this situation, $N_{QD}(Q)$ is a defect group of $c_Q$.

Set $\hat D=N_D(Q)$ and $T=Q\cap \hat D$.
Note that $T=Q\cap N_N(Q)$, that $c_Q$ is a block of $N_N(Q)${,} and that $\hat D$ is a defect group of {$c_Q$}. Set $\bar N_N(Q)=N_N(Q)/T$ and $\bar N_H(Q)=N_H(Q)/Q$. The inclusion {$N_N(Q)\leq N_H(Q)$} induces an injective homomorphism $\bar N_N(Q)\rightarrow \bar N_H(Q)$ with normal image. We identify $\bar N_N(Q)$ with a normal subgroup of $\bar N_H(Q)$.
Since $D$ is abelian, each block of $\bar N_N(Q)$ contained in the block $c_Q$ has defect group $\hat D/T$.

Assume that $D$ is not contained in $Q$. Then $Q$ is properly contained in $QD$, hence properly contained in $N_{QD}(Q)=Q\hat D$, so the quotient group $\hat D/T$ is nontrivial. But some block of $\bar N_N(Q)$ contained in the block $c_Q$ is covered by the block of $\bar N_H(Q)$ containing $\bar\psi$ while the latter has the trivial defect group. This yields a contradiction.
\end{proof}

\begin{lem}\label{Op-prime}
Let $G$ be a finite group and $b$ a block of $G$.
Let $N$ be a normal subgroup of $G$ and $c$ a $G$-invariant nilpotent block of $N$ covered by $b$.
Let $D$ be a defect group of $c$.
Then there exist a finite group $L$ containing $D$ as a normal subgroup,
a central $p'$-subgroup $Z$ of $L$, and a block $d$ of $L$ such that $b$ and $d$ are basically Morita equivalent and such that
$L/DZ \cong G/N$, $l(b)=l(d)$, and $\ell(b)=\ell(d)$.
Furthermore, in one of the following two cases, we have $l(b)=\ell(b)$.

\smallskip\noindent{\bf 5.2.1.}  The Sylow $2$-subgroups of $G/N$ are cyclic or {generalized quaternion}.

\smallskip\noindent{\bf 5.2.2.} {We have} $p=2$ and, for each odd prime $r$, the Sylow $r$-subgroups of $G/N$ are abelian.
\end{lem}

\begin{proof}
By the main theorem of \cite{KP} {(see also \cite{PZ})}, there exists a finite group $L$ containing $D$ as a normal subgroup, a central $p'$-subgroup $Z$ of $L$, and a block $d$ of $L$ such that $L/DZ \cong G/N$ and $b$ and $d$ are basically Morita equivalent. {Hence} $l(b)=l(d)$. By \cite[Corollary 1.3]{Hua}, we have $\ell(b)=\ell(d)$.

The nonabelian simple groups involved in $L$ are precisely those involved in $L/DZ =G/N$.
In order to prove $l(b)=\ell(b)$, by Sp\"ath's reduction theorem (see \cite{Sp}), it is therefore enough to verify the inductive blockwise Alperin weight condition stated in \cite[Definition 4.1]{Sp} for the nonabelian simple groups involved in $G/N$.

Suppose that Statement {\bf 5.2.1} holds.
If the Sylow \(2\)-subgroups of \(G/N\) are cyclic, then \(G/N\) is solvable, so there are no nonabelian composition factors to consider. Suppose that the Sylow $2$-subgroups of $G/N$ are {generalized quaternion}. Recall that a group $H$ is said to be involved in $G$ if there exist subgroups
$H_1 \unlhd H_2 \leq G$ such that $H_2/H_1 \cong H$. A Sylow \(2\)-subgroup of every section of \(G/N\) is involved in a generalized quaternion group and is therefore cyclic, a generalized quaternion group, a Klein four group, or a dihedral group.
By \cite[Theorem 1]{BS}, there is no finite nonabelian simple group with generalized quaternion Sylow $2$-subgroups.
Consequently, every finite nonabelian simple group involved in $G/N$ has Sylow $2$-subgroups that are either Klein four groups or dihedral groups.
According to \cite[Theorem 2]{GW} and \cite[Theorem 1]{Wal}, such finite nonabelian simple groups are precisely
the projective special linear groups $\operatorname{PSL}_2(q)$ or the alternating group $\mathfrak{A}_7$.
By \cite[Theorem 1]{FLZ} and \cite[Theorem 1.1]{M1}, the inductive blockwise Alperin weight condition holds for these simple groups.
Therefore \cite[Theorem A]{Sp} implies that the blockwise Alperin weight conjecture is true for $d$ and thus for $b$.

Suppose that Statement {\bf 5.2.2} holds. For each odd prime $r$,
the nonabelian simple groups involved in $G/N$ have abelian Sylow $r$-subgroups.
Let $S$ be a nonabelian simple group involved in $G/N$.
If $S$ is of Lie type with odd defining characteristic, then by \cite[Proposition 5.1]{SW} we have
$S\cong {\rm PSL}_2(q)$ with $r$ dividing $q$.
If $S$ is a sporadic group, \cite[Table 4]{SZ} shows that $S$ is one of $J_1$, $M_{11}$, $M_{22}$, or $M_{23}$.
Hence $S$ belongs to one of the following families: $\mathfrak{A}_n$ ($n\geq 5$),
{$\PSL_2(q)$}, groups of Lie type with defining characteristic $2$,
$J_1$, $M_{11}$, $M_{22}$, and $M_{23}$.
By \cite{Bre}, \cite[Theorem 1]{FLZ}, \cite[Theorem 1.1]{M1}, and \cite[Theorem C]{Sp},
the inductive blockwise Alperin weight condition of \cite[Definition 4.1]{Sp} holds for $S$ at the prime $2$.
Applying \cite[Theorem A]{Sp} yields $l(d)=\ell(d)$, and consequently $l(b)=\ell(b)$.
\end{proof}

\begin{proof}[\rm\bf Proof of Theorem \ref{Main1}] Let $b$ be a block of $G$ with exactly one irreducible Brauer character and with defect group $P$.
{Suppose, for a contradiction, that $\ell(b)\neq1$, and choose a counterexample \((G,b)\) such that}
$
\bigl(|G:\mathbf O_{p'}(\mathbf Z(G))|,\ |G|\bigr)
$ is {minimal} with respect to the lexicographic order.

Let $N$ be a normal subgroup of $G$, let $c$ be a block of $N$ covered by $b$, let $G_c$ be the stabilizer of $c$ under $G$-conjugation{,} and let $e$ be the Fong--Reynolds correspondent of $b$ in $G_c$. We claim that $G_c=G$. The block algebras $\mathcal{O} G b$ and $\mathcal{O} G_c e$ are basically Morita equivalent, so $l(b)=l(e)$ and $\ell(b)=\ell(e)$ (see \cite[Corollary 1.3]{Hua}). Suppose that $G_c<G$. Then $|G_c: \mathbf{O}_{p'}(\mathbf{Z}(G_c))|< |G: \mathbf{O}_{p'}(\mathbf{Z}(G))|$; otherwise, $G=G_c\mathbf{O}_{p'}(\mathbf{Z}(G))=G_c$. The {minimality of} $|G : \mathbf{O}_{p'}(\mathbf{Z}(G))|$ gives $\ell(e)=1$, hence $\ell(b)=1$.
This contradiction {proves the claim}.

Suppose that the block $c$ has trivial defect group. We claim $N\leq \mathbf{O}_{p'}(\mathbf{Z}(G))$. Then $\mathcal{O}Nc$ is a matrix algebra over $\mathcal O$.  Obviously there is a block algebra of $N\mathbf{O}_{p'}(\mathbf{Z}(G))$ covering $\mathcal{O}Nc$ and being isomorphic to $\mathcal{O} N c$.
By Lemma~\ref{Op-prime}, there exists a finite group $L$, a central $p'$-subgroup $Z$ of $L$ and a block $d$ of $L$ such that $b$ and $d$ are basically Morita equivalent and such that $L/Z \cong G/(N\mathbf{O}_{p'}(\mathbf{Z}(G)))$, $l(b)=l(d)$ and $\ell(b)=\ell(d)$. Clearly we have
$|L: \mathbf{O}_{p'}({\bf Z}(L))|\leq|G: N\mathbf{O}_{p'}(\mathbf{Z}(G))| \leq |G: \mathbf{O}_{p'}(\mathbf{Z}(G))|$.
If $|L: \mathbf{O}_{p'}({\bf Z}(L))|< |G: \mathbf{O}_{p'}(\mathbf{Z}(G))|$, then the {minimality of} $|G : \mathbf{O}_{p'}(\mathbf{Z}(G))|$ gives $\ell(d)=1$ and hence $\ell(b)=1$, {a contradiction}. Thus $|L: \mathbf{O}_{p'}({\bf Z}(L))|= |G: \mathbf{O}_{p'}(\mathbf{Z}(G))|$, $N\mathbf{O}_{p'}(\mathbf{Z}(G))=\mathbf{O}_{p'}(\mathbf{Z}(G))${,} and $N\leq \mathbf{O}_{p'}(\mathbf{Z}(G))$.
{This proves the claim.}

If \((Q, \psi)\) is a \(p\)-weight of the block $b$, then \(Q\) is radical. Set \(R=Q\mathbf O_p(G)\). Then \(N_R(Q)\) is a normal \(p\)-subgroup of \(N_G(Q)\), so
\(N_R(Q)\leq Q\). Since \(Q\leq N_R(Q)\), we have \(N_R(Q)=Q\). So $Q=R$ and
\(\mathbf O_p(G)\leq Q\). Set $\bar G=G/\mathbf O_p(G)$ and denote by $\bar T$ the image of a subgroup $T$ in $\bar G$. For any subgroup $T$ of $G$ containing $\mathbf O_p(G)$, the natural homomorphism $G\rightarrow \bar G$ induces {the} group isomorphism $$N_G(T)/T\cong N_{\bar G}(\bar T)/\bar T. $$
{Consequently},
the natural homomorphism $G\rightarrow \bar G$ induces a bijective correspondence $T\mapsto \bar T$ between radical subgroups of $G$ and $\bar G$.

Since $b$ has exactly one irreducible Brauer character, there is a unique block $\bar b$ of $\bar G$ dominated by $b$, and $\bar b$ also has exactly one irreducible Brauer character. Let $(Q, \psi)$ be a $p$-weight of $G$. By the preceding paragraph, $\mathbf O_p(G)\leq Q$. Through the isomorphism above, $\psi$ may be regarded as a character of $N_{\bar G}(\bar Q)/\bar Q$. Then $(\bar Q, \psi)$ is a $p$-weight of $\bar G$, and $(Q, \psi)\mapsto (\bar Q, \psi)$ defines a bijection between the $p$-weights of $G$ and those of $\bar G$; this bijection also identifies their respective conjugacy classes. Since $b$ has exactly one {irreducible Brauer character}, \cite[Lemma 2.4(b)]{Sp} shows that
$(Q, \psi)$ is a $p$-weight of $b$ if and only if $(\bar Q, \psi)$ is a $p$-weight of $\bar b$. In conclusion, we have $\ell(b)=\ell(\bar b)$.

We claim that ${\bf O}_p(G)$ is trivial.
We have $|\bar G : \mathbf{O}_{p'}(\mathbf{Z}(\bar G))|\leq |G : \mathbf{O}_{p'}(\mathbf{Z}(G))|$. If this inequality is strict, the {minimality of} $|G : \mathbf{O}_{p'}(\mathbf{Z}(G))|$ gives $\ell(\bar b)=1$ and hence $\ell(b)=1$, a contradiction. Therefore we have $|\bar G : \mathbf{O}_{p'}(\mathbf{Z}(\bar G))|= |G : \mathbf{O}_{p'}(\mathbf{Z}(G))|$, which forces $\mathbf O_p(G)=1$. This proves the claim.

We claim \(\mathbf E(G)\neq1\). Otherwise,
$\mathbf F^*(G)=\mathbf Z(G)$ and the self-centralizing property of \(\mathbf F^*(G)\) would give \(G=\mathbf Z(G)\), which is not a counterexample.

Let $X$ be a component of $G$.
Let $\{X_1, \dots, X_r\}$ be the $G$-orbit of $X$ under the conjugation action on the set of components such that $X_1 = X$. Set $L = X_1 X_2 \cdots X_r$,
let $b_L$ be a block of $L$ covered by $b$ and let $f$ be a block of $X$ such that $f b_L\neq 0$.
By the second paragraph of the proof, $b_L$ is $G$-stable and hence is unique. Since $X_1, X_2, \ldots, X_r$ {commute with one another}, $f$ is also unique. Moreover,
$P_L=P \cap L$ is a defect group of $b_L$ and $P_X=P \cap X$ is a defect group of $f$.

For each $i$, choose
$\alpha_i\in{\rm Inn}(G)$ such that
$\alpha_1=1$ and $X_i=\alpha_i(X)$.
Denote by $X^r$ the direct product of $r$ copies of $X$ and
define $$\pi:X^r\rightarrow L, (x_1,\ldots,x_r)\mapsto \alpha_1(x_1)\cdots\alpha_r(x_r).$$
Since distinct components commute and intersect centrally,
$\pi$ is surjective and $\ker(\pi)\leq Z(X^r)$.
As in the proof of Lemma \ref{Normal-one-orbit2}, we identify $X_i$ with $X$ and $\mathcal{O}X^r$ with the tensor product
$\mathcal{O}X\bigotimes_{\mathcal{O}}\mathcal{O}X\bigotimes_{\mathcal{O}}\cdots\bigotimes_{\mathcal{O}}\mathcal{O}X $ and then prove that {$f^{\otimes r}=f\otimes\cdots\otimes f$} is the block of $X^r$ dominating $b_L$.
Since $\mathbf{O}_p(G)=1$, we have $\mathbf{O}_p(X)=1$ and hence $\mathbf Z(X^r)$ is a $p'$-group. Consequently, {$\pi(f^{\otimes r})=b_L$} since $\ker(\pi)\leq \mathbf Z(X^r)$. Clearly $P_X^r$ is a defect group of $f^{\otimes r}$. Therefore $\pi(P_X^r)$ is a defect group of $b_L$ isomorphic to $P_X^r$. Since $P_L$ is also a defect group of $b_L$, the subgroups $\pi(P_X^r)$ and $P_L$ are conjugate in $L$. Composing $\pi$, if necessary, with the corresponding inner automorphism of $L$ does not change $\ker(\pi)$ or either of the blocks $f^{\otimes r}$ and $b_L$. We may therefore choose the identifications above so that
\(
 \pi(P_X^r)=P_L
\).

By the third paragraph of the proof, $\pi(P_X^r)$ is nontrivial{; hence} $P_X$ is nontrivial.
Indeed, if $P_L=1$, then $b_L$ has defect zero, and the claim
proved in the third paragraph, applied to $L\trianglelefteq G$, gives
$L\leq\mathbf O_{p'}(\mathbf Z(G))$. This is impossible because $L$
contains the noncentral component $X$.
Since $l(b)=1$, Lemma \ref{Normal-one-orbit2} implies that the set $\IBr(f)$  is a single \(\Aut(X)_{f}\)-orbit.
Then, by Proposition~\ref{Pro:Quai-simple} and \ref{Quasi-one}, the block $f$ is inertial and $P_X$ is abelian. Consequently for any $Q\leq P_X$, $\operatorname{dz}(N_X(Q)/Q,\, f)$ is not empty if and only if $Q=P_X$.

Let $f_0$ be the Brauer correspondent of $f$ in $N_X(P_X)$.
For any $\chi \in \operatorname{dz}(N_X(P_X)/P_X,\, f)$, the inflation of $\chi$ to $N_X(P_X)$ belongs to $f_0$
and the restriction $\chi^0$ of $\chi$ to the $p'$-elements of $N_X(P_X)$ belongs to $\IBr(f_0)$; moreover, the correspondence $\chi\mapsto \chi^0$ is a natural bijection between $\operatorname{dz}(N_X(P_X)/P_X,\, f)$ and $\IBr(f_0)$.

The kernel $\operatorname{Ker}(b)$ of the block $b$ is a normal $p'$-group and by the proof above, we may assume that $\operatorname{Ker}(b)$ is central in $G$. The quotient of $G$ by $\operatorname{Ker}(b)$ preserves irreducible Brauer characters and $p$-weights. By induction on $|G|$, we may assume that $\operatorname{Ker}(b)$ is trivial. Then it follows from \cite[Theorem~6.10]{N1} and Clifford theory of characters that $\operatorname{Ker}(f)$ is equal to $1$. The block $f$ cannot be principal; otherwise, since $\IBr(f)$ is a single \(\Aut(X)_{f}\)-orbit, it consists only of the trivial character, and $X$ would be $p$-nilpotent, a contradiction.
Thus for any $\phi\in \IBr(f)$, $\operatorname{Ker}(\phi)$ is a proper normal subgroup of $X$ and is therefore contained in $\mathbf Z(X)$. Since $\mathbf{O}_p(X)=1$, all $\operatorname{Ker}(\phi)$ for $\phi\in \IBr(f)$ are $p'$-groups. Moreover, all $\phi$ cover the same linear character $\lambda$ of $\mathbf Z(X)$ and hence all $\operatorname{Ker}(\phi)$ equal $\operatorname{Ker}(\lambda)$. Therefore, $\operatorname{Ker}(\phi)=\operatorname{Ker}(f)=1$ for every $\phi\in \IBr(f)$.

By \cite[Theorem~1.1]{HZ}, the inductive blockwise Alperin weight
condition of \cite[Definition~3.2]{KS} holds for $f$. Thus, if
$\Aut(X)_{P_X,f}$ denotes the subgroup stabilizing
$P_X$ and $f$, there is an
$\Aut(X)_{P_X,f}$-equivariant bijection
\[
 \Omega_X\colon\IBr(f)\longrightarrow
 \IBr(f_0).
\]
Moreover, for any $\phi\in \IBr(f)$, there exists a group $A = A(\phi)$ such that
$X \trianglelefteq A$, $A/C_A(X) \cong \Aut(X)_{\phi}$, $C_A(X)={\bf Z}(A)$,
$p \nmid |{\bf Z}(A)|$ and there exists a block isomorphism of modular character triples (see \cite[Definition 3.4]{MRR})
$$
(A, X, \phi) \succeq_b \bigl(N_A(P_X), N_X(P_X), \Omega_X(\phi)\bigr).
$$
Then by \cite[Lemma 3.10]{MRR}, we have
\[
(A\wr \mathfrak{S}_r, X^r, {\phi^{\otimes r}}) \succeq_b \bigl(N_A(P_X)\wr \mathfrak{S}_r, N_X(P_X)^r,
 {\Omega_X(\phi)^{\otimes r}}\bigr).
\]

\bigskip\noindent
\textbf{Step 1. Compatibility with the central quotient.}

We first claim
\[
  \operatorname{Ker}(\phi^{\otimes r})\leq {\bf Z}(X^r).
\]
Set \(K=\operatorname{Ker}(\phi^{\otimes r})\).  Then
\(K\trianglelefteq X^r\) and \(K\) is invariant under the natural
action of \(\mathfrak S_r\).
Suppose that \(K\nleq {\bf Z}(X^r)\).  Then, for some \(i\), the image
\(\operatorname{pr}_i(K)\) is a non-central normal subgroup of \(X\), where $\operatorname{pr}_i(K)$ denotes the $i$-th projection $X^r\rightarrow X$.
Since \(X\) is quasisimple, \(\operatorname{pr}_i(K)=X\).  By
\(\mathfrak S_r\)-invariance, \(\operatorname{pr}_j(K)=X\) for every
\(j\).  If \(X_j\) denotes the \(j\)-th direct factor, then
\[
  [K,X_j]\leq K\cap X_j
  \,\,\text{and}\,\,
  \operatorname{pr}_j([K,X_j])
     =[\,\operatorname{pr}_j(K),X\,]
     =[X,X]=X.
\]
Thus \(X_j\leq K\) for every \(j\), and hence \(K=X^r\).  This would
make \(\phi^{\otimes r}\) and therefore \(\phi\) trivial, contrary
to the choice of \(\phi\).  Hence
\(K\leq {\bf Z}(X^r)\), as claimed.

Since $\mathbf Z(X^r)$ is a $p'$-subgroup of $X^r$, it follows from \cite[Theorem 6.10]{N1} that we have
${\operatorname{Ker}(\phi^{\otimes r})=\operatorname{Ker}(f^{\otimes r})}$. Furthermore
\cite[Lemma 5.6.4]{F} yields
${\operatorname{Ker}(f^{\otimes r}) \leq \operatorname{Ker}(f_0^{\otimes r})}$. Hence we have
\[\operatorname{Ker}(\phi^{\otimes r}) \leq{\operatorname{Ker}(f_0^{\otimes r}) \leq \operatorname{Ker}(\Omega_X(\phi)^{\otimes r})}.
\]
Set ${Z_0 = \operatorname{Ker}(\phi^{\otimes r})}$.
For every $J \leq A \wr \mathfrak{S}_r$, set $\bar{J} = JZ_0/Z_0$ and regard $\chi \in \IBr(JZ_0)$ as a character of $\overline{J}$ if $Z_0 \leq \operatorname{Ker}(\chi)$, denoted by $\overline{\chi}$.

We are going to verify that
the hypotheses of
\cite[Lemma 3.15]{MRR} are satisfied for the triple $A\wr \mathfrak{S}_r$, $X^r$ and $Z_0$. For simplicity, set
\begin{center}$W=A\wr \mathfrak{S}_r$, $Y=X^r$, $W_0=N_A(P_X)\wr\mathfrak S_r$, $Y_0=N_X(P_X)^r$, $\theta=\phi^{\otimes r}$ and
 $\theta_0=\Omega_X(\phi)^{\otimes r}$.\end{center}
The character \(\theta\) is \(W\)-invariant, so
\(Z_0=\ker(\theta)\) is normal in \(W\).  Moreover,
\(Z_0\leq {\bf  Z}(Y)\leq Y_0\), \(p\nmid |Z_0|\) and
\(
 Z_0\leq\ker(\theta)\cap\ker(\theta_0).
\)
Thus both \(\bar\theta\) and \(\overline{\theta_0}\) are well defined by
deflation.

Suppose \(wZ_0\in C_{W/Z_0}(Y/Z_0)\). Then
\([w,Y]\leq Z_0\). Since \(Z_0\leq {\bf Z}(Y)\), the map
$
\delta_w:Y\rightarrow Z_0,
y\mapsto y^{-1}y^w
$
is a group homomorphism. Since \(Y=X^r\) is perfect and \(Z_0\)
is abelian, every homomorphism from \(Y\) to \(Z_0\) is trivial.
Thus \(\delta_w=1\) and hence \(w\in C_W(Y)\). Consequently, we have
\(
C_{W/Z_0}(Y/Z_0)=C_W(Y)/Z_0.
\)
We have now verified all the additional hypotheses of
\cite[Lemma~3.15]{MRR}: the subgroup \(Z_0\) is a central normal \(p'\)-subgroup in \(Y\) contained in both kernels of $\theta$ and $\theta_0$, and the required
centralizer equality is the one just proved.  Applying that lemma to the
wreath-product block isomorphism above therefore gives
 \[
 (\overline{W}, \overline{Y},
 \overline{\theta}) \succeq_b  \bigl(\overline{W_0},
 \overline{Y_0},
 \overline{\theta_0}\bigr).
 \]
Thus the wreath-product block isomorphism is compatible with deflation through the same central $p'$-subgroup $Z_0$ on both sides; this is the central-quotient compatibility needed in the next step.

\medskip\noindent
\textbf{Step 2. Descent to $L$ and automorphism equivariance.}

Since $\operatorname{Ker}(b)=1$, we have $\operatorname{Ker}(b_L) = 1$ and hence $Z_0={\operatorname{Ker}(f^{\otimes r}) = \operatorname{Ker}(\pi)}$.
So we may identify
$\overline{X^r}$ with $L$ via $\pi$. With this identification, we have
$\operatorname{bl}(\overline\theta)=b_L$. Since \(Z_0\) is a central \(p'\)-subgroup, \(P_X^r\) is the unique
Sylow \(p\)-subgroup of \(P_X^rZ_0\). Therefore we have
\[
\begin{aligned}
 N_L(P_L)
 &=N_{X^r/Z_0}(P_X^rZ_0/Z_0)\\
 &=N_{X^r}(P_X^rZ_0)/Z_0\\
 &=N_{X^r}(P_X^r)/Z_0
  =N_X(P_X)^r/Z_0
  =\overline{N_X(P_X)^r}.
\end{aligned}
\]
The same calculation in $\overline W=W/Z_0$ gives
\(
 \overline{W_0}=N_{\overline W}(P_L).
\)
Consequently, if $J$ is any subgroup of $\overline W$ containing $L$,
then
\(
 J\cap\overline{W_0}=N_J(P_L).
\)
Thus the subgroup on the right-hand side of the restricted triple
isomorphism below is exactly the normalizer required in
\cite[Lemma~3.6(i)]{MRR}.

Let $G_{\bar\theta}$ be the stabilizer of $\bar\theta$ in $G$ and let
\[
 \epsilon_{G_{\bar\theta}}\colon G_{\bar\theta}\rightarrow
 \Aut(L)
 \quad\text{and}\quad
 \epsilon_{\overline W}\colon\overline W\rightarrow
 \Aut(L)
\]
be the homomorphisms induced by conjugation. We are going to prove
\begingroup
\makeatletter
\let\veqno\@@leqno
\makeatother
\begin{equation}\label{eq:automorphism-image-inclusion}
 \epsilon_{G_{\bar\theta}}(G_{\bar\theta})
 \leq \epsilon_{\overline W}(\overline W).
\end{equation}
\endgroup

Fix $g\in G_{\bar\theta}$. Since $L=X_1\cdots X_r$, conjugation by $g$ permutes
these components. Thus there is a permutation
$\sigma\in\mathfrak S_r$ such that $X_i^g=X_{\sigma(i)}$ for every
$i$. For each $i$, define
\[
 \beta_i=\alpha_{\sigma(i)}^{-1}\circ c_g\circ\alpha_i
 \bigm|_X\in\Aut(X),
\]
where $c_g$ denotes conjugation by $g$. These maps define an
automorphism $\hat g$ of $X^r$ which sends the $i$-th factor to
the $\sigma(i)$-th factor via $\beta_i$ and satisfies
\begingroup
\makeatletter
\let\veqno\@@leqno
\makeatother
\begin{equation}\label{eq:lift-conjugation}
 \pi\circ\hat g=c_g\circ\pi.
\end{equation}
\endgroup

Recall that $\theta=\phi^{\otimes r}$ and that
$\theta=\bar\theta\circ\pi$ on the $p$-regular elements of $X^r$.
Since $g$ stabilizes $\bar\theta$, equation
\eqref{eq:lift-conjugation} implies that $\hat g$ stabilizes
$\theta$. Let $x$ be a $p$-regular element of $X$ and place $x$ in
the $i$-th factor of $X^r$ with the identity in all other factors.
The $\hat g$-invariance of $\theta$ then gives
\[
 \phi(x)\phi(1)^{r-1}
 =\phi(\beta_i(x))\phi(1)^{r-1}.
\]
Consequently, we have $\phi^{\beta_i}=\phi$ and hence
$\beta_i\in\Aut(X)_\phi$ for every $i$.

By the choice of $A=A(\phi)$, the conjugation homomorphism
$A\to\Aut(X)_\phi$ is surjective. Choose $a_i\in A$
inducing $\beta_i$. With the convention that
$(a_1,\ldots,a_r)\sigma\in A\wr\mathfrak S_r$ sends the $i$-th
factor to the $\sigma(i)$-th factor via $a_i$, set
\[
 w=(a_1,\ldots,a_r)\sigma\in W.
\]
The automorphism of $X^r$ induced by $w$ is precisely $\hat g$.
Since $Z_0=\operatorname{Ker}(\pi)$, the image $\bar w$ of $w$ in
$\overline W=W/Z_0$ induces the same
automorphism on $L\cong X^r/Z_0$ as $g$ does. More explicitly, for every $\mathbf x\in X^r$, we have
\[
 \epsilon_{\overline W}(\bar w)(\pi(\mathbf x))
 =\pi(\mathbf x^{\hat g})
 =\pi(\mathbf x)^g
 =\epsilon_{G_{\bar\theta}}(g)(\pi(\mathbf x))
\]
and hence
$\epsilon_{G_{\bar\theta}}(g)=\epsilon_{\overline W}(\bar w)$.
This proves
\eqref{eq:automorphism-image-inclusion}.

In particular, every automorphism of $L$ induced by an element of
$G_{\bar\theta}$ is induced by an element of $\overline W$. Hence the restriction of the quotient triple isomorphism to
$\tilde G_{\bar\theta}= \epsilon_{\overline{W}}^{-1} \bigl( \epsilon_{G_{\bar\theta}}(G_{\bar\theta}) \bigr)$, followed by the butterfly argument
\cite[Lemma~3.11]{MRR}, is compatible with the conjugation actions on $L$ and
on $N_L(P_L)$. This is the automorphism equivariance required for the descent from $X^r$ to $L$.
It follows
from \cite[Lemma~3.6(i)]{MRR} that we have
\[
({\tilde{G}_{\bar\theta}}, L, {\bar\theta}) \succeq_b
\bigl( {N_{\tilde{G}_{\bar\theta}}(P_L)}, N_L(P_L), {\overline{\theta_0}} \bigr).
\]
Applying \cite[Lemma~3.11]{MRR}, we obtain
\[
({G_{\bar\theta}}, L, {\bar\theta}) \succeq_b
\bigl( {N_{G}(P_L)_{\bar\theta_0}}, N_L(P_L), {\bar\theta_0} \bigr).
\]

\bigskip\noindent
\textbf{Step 3. Passage from \(b\) to \(b_0\) and comparison of
weight sets.}

Set \(M=N_G(P_L)\) and \(d_L=\bl(\overline{\theta_0})\). Note that $d_L$ is the Brauer correspondent of $b_L$ in $N_L(P_L)$. Since
\(\IBr(b_L)\) is a single \(G\)-orbit and then a single \(\Aut(L)_{b_L}\)-orbit, \cref{Main} and Proposition \ref{Quasi-one} show that the block $b_L$ is inertial with abelian defect groups. Then it follows from \cite[Theorem]{Zhou1} that
\(\IBr(d_L)\) is a single \(M\)-orbit.  Let \(b_0\) be the Brauer
correspondent of \(b\) in \(M\) covering \(d_L\) (see \cite[Theorem~9.28]{N1}).  Since \(P\) normalizes
\(P_L=P\cap L\), we have \(P\leq M\) and \(P\) is also a defect group of
\(b_0\).  Notice also that \(N_L(P_L)\trianglelefteq M\), that \(d_L\) is
\(M\)-invariant and that \(P_L\) is a defect group of \(d_L\).

The block isomorphism of modular character triples obtained in Step~2 induces
a bijection
\[
 \sigma_{\bar\theta}\colon
 \IBr(G_{\bar\theta}\mid\bar\theta)
 \rightarrow
 \IBr(M_{\overline{\theta_0}}\mid\overline{\theta_0})
\]
such that, if
\(\eta\in\IBr(G_{\bar\theta}\mid\bar\theta)\), then we have
\(
 \bl(\eta)^G=b\) if and only if
 \(\bl\bigl(\sigma_{\bar\theta}(\eta)\bigr)^M=b_0.
 \)
Set
\begin{center}\(
 \mathcal X_b=
 \{\eta\in\IBr(G_{\bar\theta}\mid\bar\theta)\mid
       \bl(\eta)^G=b\}
\)
and
\(
 \mathcal X_{b_0}=
 \{\eta_0\in\IBr(M_{\overline{\theta_0}}\mid\overline{\theta_0})
       \mid\bl(\eta_0)^M=b_0\}.
\)\end{center}
The block-preserving property in the definition of $\succeq_b$,
together with transitivity of block induction and the uniqueness of the
block $b_0$ of $M$ covering $d_L$, gives
\[
 \sigma_{\bar\theta}(\mathcal X_b)=\mathcal X_{b_0}.
\]

Since every
$\varphi\in\IBr(b)$ has all its $L$-constituents in
$\IBr(b_L)$ and $\IBr(b_L)$ is a single $G$-orbit, $\bar\theta$ is a
constituent of $\varphi_L$.  Clifford correspondence therefore gives
a unique $\eta\in\IBr(G_{\bar\theta}\mid\bar\theta)$ such that
$\eta^G=\varphi$.  Moreover,
$\bl(\eta)^G=\bl(\eta^G)=b$, so that $\eta\in\mathcal X_b$.
Conversely, for $\eta\in\mathcal X_b$, we have
$\bl(\eta^G)=\bl(\eta)^G=b$.  Hence Clifford correspondence restricts
to a bijection
\[
 \mathcal X_b\longrightarrow\IBr(b), \eta\longmapsto\eta^G.
\]
Similarly, since $\IBr(d_L)$ is a single $M$-orbit, the map
\[
 \mathcal X_{b_0}\rightarrow\IBr(b_0),\,
 \eta_0\mapsto\eta_0^M
\] is bijective.
Consequently, the map
\(
 \IBr(b)\rightarrow\IBr(b_0),
 \eta^G\mapsto\sigma_{\bar\theta}(\eta)^M
\)
is bijective. In particular, we have $l(b_0)=l(b)=1$.

Since \(P_L\neq1\) and \(\mathbf O_p(G)=1\), we have \(M<G\).  Moreover,
\(\mathbf Z(G)=\mathbf O_{p'}(\mathbf Z(G))\leq
\mathbf O_{p'}(\mathbf Z(M))\), and hence
\(
 |M:\mathbf O_{p'}(\mathbf Z(M))|
 <|G:\mathbf O_{p'}(\mathbf Z(G))|.
\)
Then it follows from the minimal choice of \((G,b)\) that
\(\ell(b_0)\) is equal to 1.

\bigskip Every \(b\)-weight is conjugate to one whose first component \(Q\) is contained in \(P\). Let $Q$ be a $p$-subgroup of $P$ such that $\operatorname{dz}(N_G(Q)/Q,b)$ is nonempty. Applying Lemma~\ref{Vertex} to
\(L\trianglelefteq G\) shows $P_L\leq Q$, hence we have $Q\cap L=P_L$.
Since \(L\trianglelefteq G\), \(N_G(Q)\) normalizes
\(Q\cap L=P_L\). Consequently, we have
\(
 N_G(Q)\leq M\) and \(
 N_G(Q)=N_M(Q).
\)
By transitivity of block induction and the uniqueness of \(b_0\),
for every defect-zero character \(\chi\) of $N_G(Q)/Q$, we have
\(
 \bl(\chi)^G=b\) if and only if
 \(\bl(\chi)^M=b_0.\)
Therefore we have
\[
 \operatorname{dz}(N_G(Q)/Q,b)
 =\operatorname{dz}(N_M(Q)/Q,b_0).
\]
Since $\ell(b_0)=1$, we have $Q=P$ and any $p$-weights
$(Q, \chi)$ and $(Q, \psi)$ of $b$ with
$\chi, \psi\in \operatorname{dz}(N_G(Q)/Q,b)$ have to be $M$- and hence
$G$-conjugate to each other. Therefore we have \(\ell(b)=1\).
This is contrary to the choice of \((G,b)\).
\end{proof}

We isolate the induction step from the proof of
\cite[Proposition~5]{Wa}. This makes explicit the part of
Watanabe's argument that will be used below.

\begin{lem}\label{Watanabe-step}
Let $b$ be a block of a finite group $H$, fix a maximal $b$-Brauer
pair $(S,f_S)$, and, for every $R\leq S$, let $(R,f_R)$ be the unique
$b$-Brauer pair contained in $(S,f_S)$. Set
\[
 H_R=N_H(R,f_R),\qquad C_R=C_H(R),\qquad {b_R=(f_R)^{H_R}}.
\]
Assume that $f_R$ is nilpotent for every $1\neq R\leq S$ and $l(b)=1$.

\smallskip\noindent{\bf 5.3.1.} Then
$N_H(S,f_S)/C_H(S)$
is a $p$-group.

\smallskip\noindent{\bf 5.3.2.} Let $1<R<S$ and assume that
$N_H(T,f_T)/C_H(T)$
is a $p$-group for every $b$-Brauer pair $(T,f_T)$ with
$|T|>|R|$. Then
$H_R/C_R$ satisfies the following alternative:
\begin{itemize}
\item If $p$ is odd, the Sylow $2$-subgroups of $H_R/C_R$ are cyclic
or generalized quaternion;
\item If $p=2$, the Sylow $r$-subgroups of $H_R/C_R$ are cyclic for
every odd prime $r$.
\end{itemize}
In addition, if the blockwise Alperin weight conjecture holds for
$b_R$, then $H_R/C_R$ is a $p$-group.
\end{lem}

\begin{proof}
These are the induction ingredients in the proof of
\cite[Proposition~5]{Wa}. We recall the points needed here.

Set $
 E=N_H(S,f_S)/S C_H(S)$. For $1\neq\pi\in {\bf Z}(S)$, the nilpotency of $f_{\langle\pi\rangle}$ gives
\(
 N_H(S,f_S)\cap C_H(\pi)=S C_H(S).
\)
Thus $E$ acts fixed-point-freely on ${\bf Z}(S)\setminus\{1\}$. The extension and Clifford-theoretic argument in the second paragraph of \cite[proof of Proposition~5, pp.~131--132]{Wa} then gives $l(b_S)=k(E)$, where $k(E)$ denotes the number of conjugacy classes of $E$.
Note that $l(b_S)$ is precisely the
number of $H$-conjugacy classes of $p$-weights of $b$ whose first component is
$S$.
If $l(b)=1$, by \cref{Main1}, there is exactly one $H$-conjugacy class of $p$-weights of $b$ with a representative whose first component is $S$. Consequently, we have $k(E)=1$ and hence $E=1$, so $N_H(S,f_S)/C_H(S)$ is a $p$-group. This
proves {\bf 5.3.1}.

If $q\neq p$ and $K$ is a Sylow $q$-subgroup of $H_R/C_R$, then the nilpotency of $f_{\langle\pi\rangle}$ gives $C_K(\pi)=1$ for every $1\neq\pi\in {\bf Z}(R)$. Thus $K$ acts fixed-point-freely on ${\bf Z}(R)\setminus\{1\}$. The third paragraph of \cite[proof of Proposition~5, pp.~131--132]{Wa} uses precisely this fact to obtain the stated Sylow-subgroup alternatives. Then by its extension argument, $l(b_R)$ is equal to the number of $p$-regular conjugacy classes of $H_R/C_R$.

Finally, it is known that $b_S=f_S^{N_H(S,f_S)}$ has a simple module with vertex $S$. Thus the local block $b_S$ yields one $H$-conjugacy class of $p$-weights of $b$ with first component $S$. Since $l(b)=1$, Theorem~\ref{Main1} implies that $b_R$ yields no $H$-conjugacy class of $p$-weights of $b$ with first component $R$. This allows us to
apply its vertex counting argument in the last paragraph of
\cite[proof of Proposition~5, pp.~131--132]{Wa}.
Let $R\leq D\leq S$ be a defect group
of $b_R$. For $R<T\leq D$, the induction hypothesis and the nilpotency
of $f_T$, together with the vertex counting argument, yields that only nonzero contributions to $\ell((f_R)^{N_H(R, f_R)\cap N_H(T, f_T)})$ come from the case $T=D$. Since the blockwise Alperin
weight conjecture holds for $b_R$, \cite[Lemma~4(ii)]{Wa} yields
$l(b_R)=1$. The preceding paragraph now shows that $H_R/C_R$ has only
one $p$-regular conjugacy class, and hence is a $p$-group. This proves
the final assertion of {\bf 5.3.2}.
\end{proof}

\begin{proof}[\rm\bf Proof of Theorem \ref{Main2}] We prove by induction on $|G|$ that $e_Q$ is nilpotent for any $ Q\leq P$.
Fix a maximal $b$-Brauer pair $(P,e_P)$. For each $Q\leq P$, let
$(Q,e_Q)$ denote the unique $b$-Brauer pair contained in $(P,e_P)$
and set $H_Q=N_G(Q,e_Q)$, $C_Q=C_G(Q)$ and $c_Q=(e_Q)^{H_Q}$.
If $P=1$, then $b$ has defect zero and is nilpotent. We may therefore
assume $P\neq1$.

We claim that $e_Q$ is nilpotent for any $1< Q\leq P$.
If \(C_G(Q)<G\), every \(e_Q\)-Brauer pair \((R,f_R)\) gives the \(b\)-Brauer pair \((QR,f_R)\), so \(l(f_R)=1\); induction on \(|G|\) implies that \(e_Q\) is nilpotent.
If \(C_G(Q)=G\), then \(Q\leq {\bf Z}(G)\) and \(e_Q=b\). The hypothesis descends to the block \(\bar b\) of \(G/Q\) dominated by \(b\).
Therefore \(\bar b\) is nilpotent by induction and hence \(b=e_Q\) is nilpotent by \cite[Lemma~2.5]{AE1}.
This proves the claim.

It remains to prove that $b$ is nilpotent.
We claim that
$H_Q/C_Q$ is a $p$-group
for every nontrivial subgroup $Q\leq P$.
Since $(1,b)$ is a $b$-Brauer pair, the hypothesis gives $l(b)=1$.
By Theorem~\ref{Main1}, the blockwise Alperin weight conjecture holds
for $b$. Statement {\bf 5.3.1} applied to $(P,e_P)$
shows that $N_G(P,e_P)/C_G(P)$ is a $p$-group.
Let $1<Q<P$. Assume that $H_D/C_D$ is a $p$-group for every
$b$-Brauer pair $(D,e_D)$ with $|D|>|Q|$. By Statement
{\bf 5.3.2}, $H_Q/C_Q$ has cyclic or generalized
quaternion Sylow $2$-subgroups when $p$ is odd, whereas, when $p=2$,
all its Sylow subgroups at odd primes are cyclic. {Hence},
Lemma~\ref{Op-prime} applies to the normal subgroup
$C_Q\trianglelefteq H_Q$, the $H_Q$-stable nilpotent block $e_Q$ of
$C_Q$ and the covering block $c_Q$ of $H_Q$. Consequently, the blockwise Alperin weight conjecture
holds for $c_Q$. The final assertion of Statement
{\bf 5.3.2}
now yields that $H_Q/C_Q$ is a $p$-group.
In conclusion, $N_G(Q,e_Q)/C_G(Q)$ is a $p$-group for every nontrivial
$b$-Brauer pair $(Q,e_Q)$. The local characterization of nilpotent blocks \cite{P-1} now implies that $b$ is nilpotent.
\end{proof}


\end{document}